\documentclass{article}

\usepackage{lmodern,bbm}
\usepackage{amsmath}
\usepackage{amssymb}
\usepackage{mathrsfs}
\usepackage{esint}
\usepackage{enumitem}
\usepackage{amsthm}
\usepackage{cleveref}
\usepackage{cite}
\usepackage{url}
\usepackage{geometry}

\newtheorem{theorem}{Theorem}[section]
\newtheorem{lemma}[theorem]{Lemma}

\theoremstyle{definition}
\newtheorem{definition}[theorem]{Definition}

\newtheorem{corollary}[theorem]{Corollary}
\newtheorem{proposition}[theorem]{Proposition}
\theoremstyle{remark}
\newtheorem{remark}[theorem]{Remark}

\numberwithin{equation}{section}

\DeclareMathOperator\R{\mathbb{R}}
\DeclareMathOperator\N{\mathbb{N}}
\DeclareMathOperator\argmin{argmin}
\newcommand*\prob{\mathop{}\!\mathscr{P}}
\newcommand{\p}{\varrho}
\newcommand*\diff{\mathop{}\!\mathrm{d}}

\title{Doubly nonlinear diffusive PDEs: new existence results via generalized Wasserstein gradient flows}

\author{Thibault Caillet\thanks{Universite Claude Bernard Lyon 1, ICJ UMR5208, CNRS, Ecole Centrale de Lyon, INSA Lyon, Universit\'e Jean Monnet,
69622 Villeurbanne, France. ({caillet@math.univ-lyon1.fr})}
\and Filippo Santambrogio\thanks{Universite Claude Bernard Lyon 1, ICJ UMR5208, CNRS, Ecole Centrale de Lyon, INSA Lyon, Universit\'e Jean Monnet,
69622 Villeurbanne, France. ({santambrogio@math.univ-lyon1.fr)}}
}

\begin{document}

\maketitle

\begin{abstract} We prove an existence result for a large class of PDEs with a nonlinear Wasserstein gradient flow structure. We use the classical theory of Wasserstein gradient flow to derive an EDI formulation of our PDE and prove that under some integrability assumptions on the initial condition the PDE is satisfied in the sense of distributions. \end{abstract}

\section{Introduction}

This manuscript is concerned with parabolic non-linear equations of the form
\begin{equation*}\partial_t\rho=\Delta_q(h(\rho)),\end{equation*}
where $h:\R_+\to\R$ is an increasing function and $q>1$ a fixed exponent, by viewing it as a generalized gradient flow in the Wasserstein space $\mathbb W_p$ (where $p$ and $q$ are conjugate exponents). In order to introduce this approach, we will start our presentation from the easiest case of classical gradient flows in Hilbert spaces and slowly arrive to the case we aim to focus at.

The most classical theory of gradient flows is concerned with evolution equations of the form 
\begin{equation*}
x'(t)=-\nabla \mathscr{F}(x(t)),
\end{equation*}
for instance when $x$ is a curve valued in a certain Hilbert space $H$ and $\mathscr{F}:H\to \R$ is a given function on such a space. It is well-known that a natural time-discretization of this equation takes the form of a sequence of optimization problems, where we define, for a given time step $\tau>0$, a sequence $(x_k)_k$ of points in $H$ via 
\begin{equation*}
x_{k+1}=\argmin\left\{\mathscr{F}(x)+\frac{||x-x_k||^2}{2\tau}\right\}.
\end{equation*}
This corresponds to the well-known implicit Euler scheme for the above equation, but has the advantage that it admits a formulation in metric spaces, where the norm $||x-x_k||$ is replaced by a generic distance. Such an iterated minimization procedure is usually called {\it minimizing movement scheme} (see \cite{DeG}). 

If one replaces the power $2$ in these optimization problems with another power, the natural scaling is the following
\begin{equation*}
x_{k+1}=\argmin\left\{\mathscr{F}(x)+\frac\tau p\left(\frac{||x-x_k||}{\tau}\right)^p\right\}
\end{equation*}
and leads to a solution of 
\begin{equation*}
x'(t)=-(\nabla \mathscr{F}(x(t)))^{q-1},
\end{equation*}
where for a vector $v$ and an exponent $\alpha>0$ we denote by $v^\alpha$ the vector whose direction is the same as that of $v$ and whose norm is $||v||^\alpha$, i.e. $||v||^{\alpha-1}v$, and the exponent $q$ is the conjugate exponent of $p$. As a consequence, we do not face a true gradient flow, but a different variational evolution that we call {\it generalized gradient flow}. A similar construction could be made in metric spaces, of course, where the notion of (generalized) gradient flow should be suitably defined (and we refer to \cite{AGS} or \cite{San-surv} for the details).

A particular case of interest is that of the Wasserstein spaces of probability measures endowed with the distances $W_p$ induced by optimal transport. These distances are now used in a variety of contexts, from PDEs to machine learning, and we will not provide here extra details on them. The main facts that we will use will be presented in Section 2.2, and the interested reader is encouraged to refer to \cite{VillaniOT} or \cite{OTAM}. Each distance $W_p$ being defined as the $p$-root of a transport cost, it is natural to use each $W_p$ only in the framework of a generalized gradient flow with exponent $p$.

The procedure consisting in studying the limits (as $\tau\to 0$) of the following iterated optimization problems
\begin{equation*}\rho_{k+1}=\argmin\left\{\mathscr{F}(\rho)+\frac{W_2^2(\rho,\rho_k)}{2\tau}\right\}\end{equation*}
has been introduced in \cite{JKO}, is now known as 
{\it Jordan-Kinderleher-Otto} scheme, and is widely used to study PDEs of the form
\begin{equation*}\partial_t\rho=\nabla\cdot\left(\rho\nabla \left(\frac{\delta \mathscr{F}}{\delta\rho}\right)\right),\end{equation*}
where $\frac{\delta \mathscr{F}}{\delta\rho}$ denotes the first variation of the functional $\mathscr{F}$. When $\mathscr{F}(\rho)=\int f(\rho(x))dx$ this gives rise to a parabolic PDE of the form
\begin{equation*}\partial_t\rho=\nabla\cdot\left(\rho\nabla (f'(\rho))\right).\end{equation*}

In this work, we investigate instead the general case $p\neq 2$. This consists in considering the iterated optimization problems
\begin{equation}
\label{jkoscheme}
\rho_{k+1}=\argmin\left\{\mathscr{F}(\rho)+\frac{W_p^p(\rho,\rho_k)}{p\tau^{p-1}}\right\},
\end{equation}
that we will call $p$-JKO scheme. The PDE has now  the form
\begin{equation*}\partial_t\rho=\nabla\cdot\left(\rho\left(\nabla \left(\frac{\delta F}{\delta\rho}\right)\right)^{q-1}\right).\end{equation*}
When $\mathscr{F}$ has the form $\int f(\rho(x))dx$ for a convex integrand $f$, this equation becomes the desired one
\begin{equation*}\partial_t\rho=\Delta_q(h(\rho)),\end{equation*}
where $h$ is such that $h'(s)=s^{p-1}f''(s)$ and, again, $p$ and $q$ are conjugate exponents. Equivalently, one can also write this as 
\begin{equation*}\partial_t b(u)=\Delta_q(u)\end{equation*}
with $b=h^{-1}$. This is why these equations are called {\it doubly non-linear}. They are anyway quasilinear, since the right-hand side is linear w.r.t. the second derivatives in space. However, the nonlinearity in 
terms of the gradient is source of difficulty. 

Among the first studies on these equations we mention \cite{AltLuc} where strong structural assumptions were required to prove existence and uniqueness, but we are more interested in the works which already attacked these equations in terms of generalized gradient flows and $p$-JKO schemes. the founding papers on this variational approach to these equations are \cite{Agueh} and \cite{Otto}, where it is proven that the discrete solutions obtained by the $p$-JKO scheme converge, under some conditions, to a solution of the PDE, in weak sense. 

Both \cite{Agueh} and \cite{Otto} use estimates obtained from the scheme in order to prove the necessary compactness of both the density $\rho$ and its gradient $\nabla \rho$ and pass to the limit the discrete optimality conditions at each step. They require some assumptions on the function $f$: \cite{Otto} only considers the case of powers, i.e. $f(z)=\frac{z^m}{m-1}$, and requires $m>0$; \cite{Agueh} requires $f$ to satisfy McCann's condition for geodesic convexity (see \cite{McC}; this condition is detailed in our assumption {\bf (H5)} in \cref{section2.1}). These restrictions on $f$ are disappointing since the case of the parabolic $q$-Laplacian equation
$\partial_t\rho=\Delta_q\rho$ requires to use a function $f$ such that $f''(z)=cz^{-p}$, i.e. $f(z)\approx z^{2-p}$, which could be out of the set of assumptions of these papers. 

Our goal in the present paper is to generalize the results to more generic convex functions $f$, and at the same time to change the strategy of the proof. Instead of relying on fine arguments for the strong convergence of the gradient, we sill make use of the {\it Energy-Dissipation Principle} and prove an {\it Energy Dissipation Inequality on the limit}. The idea, well-known in gradient flows (see, for instance \cite{AGS}) is the following: a curve $t\mapsto x(t)$ is a solution of $x'=-\nabla \mathscr{F}(x)$ on $[0,T]$ if and only if we have
\begin{equation}\label{EDI1}
    \mathscr{F}(x(T))+\int_0^T (\frac 12 ||x'(t)||^2+\frac 12||\nabla \mathscr{F}(x(t))||^2)\diff t\leq \mathscr{F}(x(0)).
\end{equation}
This can be seen from the fact that the chain rule provides $\frac{d}{dt}(\mathscr{F}(x(t))=\nabla \mathscr{F}(x(t))\cdot x'(t)$ and, by Young's inequality, the opposite inequality of \eqref{EDI1} is always true and the only equality case is given by $x'=-\nabla F(x)$.

Analogously, we have $x'(t)=-(\nabla \mathscr{F}(x(t)))^{q-1}$ if and only if we have 
 $\mathscr{F}(x(T))+\int_0^T (\frac 1p ||x'(t)||^p+\frac 1q||\nabla \mathscr{F}(x(t))||^q)\diff t\leq \mathscr{F}(x(0))$. The corresponding condition in the Wasserstein space when considering $\mathscr{F}(\rho):=\int f(\rho(x))dx$ would be: a curve $t\mapsto\rho_t$, solution of $\partial_t\rho+\nabla\cdot(\rho v)=0$, is a solution of the PDE if and only if we have
\begin{equation*}
\mathscr{F}(\rho(T))+\frac1p\int_0^T \int_\Omega\rho|v|^p\diff x \diff t+\frac 1q \int_0^T\int_\Omega \rho|\nabla (f'(\rho))|^q \diff x \diff t\le\mathcal \mathscr{F}(\rho(0)).
\end{equation*}
Proving such an inequality is possible, up to some technicalities, using suitable estimates on the the discrete scheme and semicontinuity arguments. The important point is that it is in general hard to prove that this characterizes the solutions of the PDE, i.e. the condition $v=-\nabla f'(\rho)$, since it is hard to prove that we do have $\frac{\diff}{\diff t}\mathscr{F}(\rho_t)=\int\nabla(f'(\rho_t))\cdot v_td\rho_t$, i.e. proving a chain rule in the Wasserstein space.

The idea of using a $p$-$q$ version of the EDI condition has been discussed in \cite{AGS}, at the metric level as well as in the setting of curves in Wasserstein space, and applied mostly to the case of geodesically convex functionals where a chain rule is proved. 
There are many other papers where the EDI condition is used to characterize solutions of these doubly nonlinear PDEs, for example we can cite \cite{CancesMatthesNabetRott}, whose goal is to prove the convergence of a numerical scheme. Also one can cite \cite{RMSDoublyNonLinear} where more general convex duality pairs are allowed, instead of simply $p$-$q$ powers.

 One of the hard tasks of the present work is to remove this assumption of geodesic convexity on the driving functional while still being able to obtain a suitable chain rule, in the case of bounded solutions and for $p=2$, this has also been discussed in \cite{ALSStabilty}.  Here, this is done by decomposing, after some approximation, an arbitrary function $f$ into the difference of two functions satisfying McCann's condition. If we do not need to perform such a decomposition, i.e. if $f$ itself satisfies McCann's condition, then the only assumption that we will require on the initial datum $\rho_0$ will be $\int f(\rho_0)<+\infty$, which is natural. If not, then we need a certain $L^\alpha$ summability, for an exponent $\alpha$ depending on the dimension $d$, on the exponent $p$, and on the function $f$. 

This condition on the integrability of the initial datum is an improvement compared to \cite{Agueh}, where most of the paper is presented under the assumption that $\rho_0$ is bounded from above and below by strictly positive constants, and then completed by a discussion on how to remove the lower bounds by approximation and then the upper bound by replacing it with $\rho_0\in L^p$ (the same $p$ as the the exponent in the Wasserstein space). Our exponent $\alpha$ could be smaller than $p$ but, in particular, we do not need it when $f$ satisfies McCann's condition, as it is the case in \cite{Agueh}.

We stress on the other hand two points where our paper is less general than the previous ones, essentially for the sake of simplicity:
\begin{itemize}
    \item After performing very long computations we decided to stick to the case $\mathscr{F}(\rho)=\int f(\rho(x))dx$ and ignore the more general case $\mathscr{F}(\rho)=\int f(\rho(x))dx+\int Vd\rho$ where a potential energy is added. Such a case is, instead, considered in \cite{Agueh,CancesMatthesNabetRott}. Yet, it makes the computations much harder in the case where $\mathscr{F}$ is not geodesically convex, and we believe the interest for these equations is limited, because, differently from the case $p=2$, the gradient of the potential does not act as a linear drift in the equation.
    \item We also decided to ignore the case where the Wasserstein distance $W_p$ are replaced by transport costs which have not the homogeneity of a power. Indeed, \cite{Agueh} considers a JKO-like scheme obtained by looking at
    \begin{equation*}
    \rho_{k+1}=\argmin\left\{\mathscr{F}(\rho)+\tau \mathcal T_{c,\tau}(\rho,\rho_k)\right\},
    \end{equation*}
    where $T_{c,\tau}$ is the minimal transport cost associated with the cost $(x,y)\mapsto c((x-y)/\tau)$. In this case the equation we obtain is
\begin{equation*}
\partial_t\rho=\nabla\cdot\left(\rho\nabla c^*\left(\frac{\delta F}{\delta\rho}\right)\right).
\end{equation*}
For an example of a very interesting PDE with this structure we cite \cite{McCPue} on the so-called {\it relativistic heat equation}, where $c(v)=1-\sqrt{1-|v|^2}$.
 \end{itemize}

 The main result of our paper is the following theorem, where we use the notation $L_f(z)=zf'(z) - f(z)$, so that the function $L_f$ satisfies $L_f'(z)=zf''(z)$ and, formally, $\nabla L_f(\rho)=\rho\nabla (f'(\rho))$.
\begin{theorem}
\label{maintheorem}
    Let $f:\R_+ \to \R \cup \{+\infty\}$ be a l.s.c.\ convex function such that $f\in C^2 (0,+\infty)$, $f''>0$, and $f''(z)\geq Cz^{-\theta}$ when $z$ is large enough for some $\theta\in\R\cup\{+\infty\}$. Let $\p_0 \in L^1 (\Omega)$ be a probability density satisfying
     \begin{align*}
        \int_\Omega f(\p_0) \diff x < \infty.
    \end{align*}
If $f$ does not satisfy McCann's condition, set $\alpha =2- q(1+\frac{1}{d})+ \theta (q-1)$, and assume as well $\p_0\in L^\alpha$ whenever $\alpha>1$, or $\int_\Omega \p_0 \log \p_0 \diff x < + \infty $ if $ \alpha = 1$.

Then there exists $\p \in C^0 ([0,T], \mathbb{W}_p (\Omega))$, obtained as any limit as $\tau \to 0$ of suitably interpolated steps of the $p$-JKO scheme \cref{jkoscheme} started from $\p_0$, such that $L_f(\p) \in W^{1,1} (\Omega)$ for a.e.\ $t$, and there exists $v\in L^p(\p_t \diff x \diff t)$ such that
    \begin{align}
    \label{PDE}
        \begin{cases}
            \partial_t \p_t + \nabla \cdot \left(\p_t v_t\right) =0 \textnormal{ in the sense of distributions,}\\
            v_t=-\left(\frac{\nabla L_f (\p_t)}{\p_t}\right)^{q-1} \; \p_t \textnormal{ a.e.\ for a.e.\ } t, \\
            \p(0)=\p_0. \\
        \end{cases}
    \end{align}
Moreover, it satisfies the Energy Dissipation Inequality
\begin{equation}
\label{edi}
     \mathscr{F}(\p_0) \geq \mathscr{F}(\p_{T}) + \frac{1}{q} \int_{0} ^{T }\int_\Omega \left| \frac{\nabla L_f(\p_t )}{\p_t } \right|^q \diff \p_t \diff t + \frac{1}{p} \int_{0} ^{T } \int_\Omega \left| v _ t \right|^p \diff \p_t \diff t .
\end{equation}
Finally, any curve satisfying \cref{edi} such that 
\begin{align*}
        &\int_{\Omega_T} |\nabla h(\p)|^q \diff x \diff t < +\infty,
    \end{align*}
    where $h$ is any function such that $h'(z)=z^{\frac{\alpha -1}{q}} f''(z)^\frac{1}{p}$ is also a solution of \cref{PDE} and actually satisfies equality in \cref{edi}.
\end{theorem}

 We underline that the summability assumption on $\p_0$ (which is void if $\alpha<1$) only plays a role in case $f$ does not satisfy McCann's condition. Since McCann's condition is satisfied for every convex function when $d=1$, the discussion about the parameter $\alpha$ is only meaningful when $d \geq 2$.

The parameter $\alpha$ is used in our proof to derive some regularity for the limit curve we construct via the $p$-JKO scheme, using an entropy inequality obtained by use of the Flow Interchange Technique (see \cref{SectionEstimates}). To give a few examples of how the integrability assumption $\p_0 \in L^\alpha$ is used in practice, one can notice that if $f$ is only assumed to satisfy $f''>0$, one has to pick $\alpha = + \infty$ so that the initial condition has to be bounded. One might wonder whether there are cases where finiteness of the initial energy implies that $\p_0 \in L^\alpha$ for the corresponding $\alpha$, which would mean that there is no additional integrability assumption even in the non geodesically convex case. Indeed if for large $z$ one has $f''(z) \geq C z^{-\theta}$, then for large $z$ one also has $f(z) \geq Cz^{2-\theta}$; in the case where $1 + \frac{1}{d}\geq \theta$, we have $2- \theta \geq \alpha$ and finiteness of the energy implies $\p_0 \in L^\alpha$. Yet, in this case we have $\alpha \leq 1 - \frac{1}{d}<1$ so that the $\alpha$-integrability condition in \cref{maintheorem} is void anyway.
As stated earlier, the difficult part of our proof lies in proving a suitable chain rule for functionals of absolutely continuous curves in Wasserstein space. We will prove the following, which we state here because it might be of independent interest and have other applications :

\begin{theorem}[Chain Rule]
\label{ChainRule1}
    Let $(\p_t)_{t\in [0,T]} \in \textnormal{AC}^p (\mathbb{W}_p (\Omega))$ with velocity vector field $(v_t)_{t\in [0,T]}$, so that it solves the continuity equation
    \begin{equation*}
        \partial_t \p_t + \nabla \cdot ( \p_t v_t) =0.
    \end{equation*}
    If $\p \in \mathcal{M}([0,T]\times \Omega)$ is absolutely continuous with respect to the Lebesgue measure, $L_f (\p_t) \in W^{1,1}(\Omega)$ for a.e.\ $t \in [0,T]$, and if
    \begin{align*}
        &\int_{\Omega_T} \left| \frac{\nabla L_f(\p_t) }{\p_t} \right|^q \diff \p_t \diff t < \infty,
    \end{align*}
    and, if $f$ does not satisfy McCann's condition,
    \begin{align*}
        &\int_{\Omega_T} |\nabla h(\p)|^q \diff x \diff t < +\infty,
    \end{align*}
    where $h$ is any function satisfying $h'(z)=z^{\frac{\alpha -1}{q}} f''(z)^\frac{1}{p}$, and $\alpha =2- q(1+\frac{1}{d})+ \theta (q-1)$, then
    \begin{align*}
        \mathscr{F} (\p_T) - \mathscr{F} (\p_0) = \int_{\Omega_T} \frac{\nabla L_{f}(\p)}{\p}v \diff \p \diff t.
    \end{align*}
\end{theorem}

 To achieve the goal of proving the above theorem, the paper is organized as follows. Section 2 is devoted to some preliminaries: notation (2.1), bases of optimal transport, including some considerations on the McCann's condition that we mentioned many times so far (2.2), and the details of an approximation procedure that we will use for every step of the $p$-JKO scheme, i.e. replacing $f$ with $f_\varepsilon$ defined via $f_\varepsilon(z)=f(z)+\varepsilon z\log z$. In section 3 we prove that the solutions obtained in the iterated minimization converge, up to subsequences, to a curve $\rho_t$ satisfying the EDI condition. This requires suitable interpolations of the discrete sequence (3.1). The proof of the compactness and of the limit inequality is contained in Section 3.2, and in Section 3.3 we prove additional integrability properties on the limit curve.

 Section 4 is devoted to the hard task of proving the chain rule and is divided into several steps, where we prove the desired results on arbitrary curves satisfying some integrability assumptions. Since the curve obtained in Section 3 satisfies them, at the end of the section we can prove \cref{maintheorem}.

 We also prove additional estimates on the solution constructed as a limit of the $p$-JKO scheme, in particular in terms of its BV norm. This is a consequence of the so-called {\it five gradients inequality} which as been recently generalized to the case of the transport cost $|x-y|^p$ in \cite{FiveGradsIneq}. This is the object of Section 5.

 A short appendix concludes the paper, where a technical proof of a chain rule result is detailed.

\section{Notations and Preliminaries}

\subsection{Notations}
\label{section2.1}
Below is a list of notations we will use in the sequel :
\begin{itemize}[label=-]
    \item $d\in \N^*$ is the space dimension,
    \item $\Omega$ is an open convex bounded subset of $\R^d$,
    \item $T>0$ is a set time horizon and $\Omega_T=\Omega\times [0,T]$,
    \item $p>1$ and $q>1$ are dual integrability exponents : $\frac{1}{p}+\frac{1}{q}=1$,
    \item $W_p$ is the $p$-Wasserstein distance.
    \item If $z\in \R^d$ is any vector and $\alpha >1$, we write $z^\alpha =|z|^{\alpha-1}z$
    \item For a differentiable function $f:\R_+ \to \R$, $L_f$ is the function defined by $L_f(z)=zf'(z) - f(z)$
    \item For a function $f:\R_+ \to \R$, $\mathscr{F} : \prob (\Omega) \to \R \cup \{+\infty\}$ is its associated functional according to \cref{deffunctional}. We use similarily $\mathscr{H},\mathscr{\Tilde{F}},\mathscr{\Tilde{F}}_1,\mathscr{\Tilde{F}}_2, \mathscr{L}$ for the functionals associated with the functions $h,\Tilde{f},\Tilde{f_1},\Tilde{f_2}, \ell$ respectively.
\end{itemize}
For a function $f : \R_+ \to \R$, we define the following assumptions which we will often use in this paper :
\begin{itemize}
    \item (\textbf{H1}) $f$ is convex and lower semi-continuous 
    \item (\textbf{H2}) $f\in C^2((0,+\infty))$ 
    \item (\textbf{H3}) $f'' >0$ 
    \item (\textbf{H4}) $f''(z) \geq C z^{-\theta}$ when $z$ is large enough
    \item (\textbf{H5}) (McCann's condition) $f(0)=0$ and $(0, +\infty) \ni s \mapsto s^d f(s^{-d})$ is convex and non increasing.
\end{itemize}
(\textbf{H1}),(\textbf{H2}) and (\textbf{H3}) will be used throughout the paper

\subsection{Preliminaries on Optimal transport, Gradient Flows and Geodesic Convexity}

We refer to \cite{AGS,OTAM,VillaniO&N} for the general theory of optimal transport and its application to gradient flows, and will compile below a selection of helpful facts we will use in the sequel. 

\begin{theorem}
    \label{everythingtheorem}
    Let $\p, g \in \prob(\Omega)$ be two probabilities on $\Omega$. The following statements are classical :
    \begin{enumerate}
        \item The problem
        \begin{align*}
            W_p ^p (\mu,\nu) := \min \left\{\int_{\Omega\times \Omega} |x-y|^p \diff \gamma \quad ; \quad \gamma \in \Pi(\mu,\nu) \right\},
        \end{align*}
        where $\Pi(\mu,\nu)$ is the set of probabilities on $\Omega\times \Omega$ with first marginal $\mu$ and second marginal $\nu$, admits a solution $\gamma^*$. If $\mu=\p \diff x$, with $\p \in L^1 (\Omega)$, then the solution is unique, and given by $\gamma^*=(\textnormal{id},T)_\# \p$ for some $T: \Omega \to \Omega$ which is called the optimal transport map. 
        \item We have 
        {\small
        \begin{align*}
            \frac{1}{p} W_p ^p (\mu, \nu) = \max \left\{ \int_\Omega \varphi \diff \mu + \int_\Omega \psi \diff \nu \ ; \ \varphi(x) + \psi(y) \leq \frac{1}{p} |x-y|^p  \; \forall x,y\in \Omega \right\}.
        \end{align*}
        }
        The optimal $\varphi$ we call Kantorovich potentials, and they can be taken to be $c$-concave, meaning of the form
        \begin{align*}
            &\varphi(x)=\inf_{y\in \Omega} \frac{1}{p} |x-y|^p - \psi(y),
            &\psi(y)=\varphi^c (y):=\inf_{x\in \Omega} \frac{1}{p} |x-y|^p - \varphi(x),
        \end{align*}
        from which one can prove that $\varphi$ and $\varphi^c$ are Lipschitz. 
        \item The optimal transport map $T$ is given by
        \begin{align*}
            T(x) = x- (\nabla \varphi (x))^{q-1},
        \end{align*}
        where $\varphi$ is any Kantorovich potential. It is well defined almost everywhere since a Lipschitz function is differentiable almost everywhere.
        \item $W_p$ is a distance on $\prob(\Omega)$ which metrizes the weak convergence of measures. Endowed with this metric, we call $\prob(\Omega)$ the Wasserstein Space, which we denote by $\mathbb{W}_p (\Omega)$
        \item $\mathbb{W}_p(\Omega)$ is a geodesic space, and the geodesic between $\mu$ and $\nu$ is given by $\mu_t:= (\pi_t)_\# \gamma$, where $\pi_t (x,y)=(1-t)x + ty$ and $\gamma$ is an optimal transport plan between $\mu$ and $\nu$. If $\gamma$ is given by an optimal transport map $T$, then $\mu_t =\left((1-t)\textnormal{id}+ tT\right)_\# \mu$.
    \end{enumerate}
\end{theorem}
We will now define a functional which will be useful in the sequel. We take $p \in (1, +\infty)$, and set $K_p =  \left\{ (a,b) \in \R \times \R ^d; a + \frac{1}{q} | b|^q \leq 0 \right\} $, and for $(t,x) \in \R_+ \times \R^d,$
\begin{align*}
    B_p(t,x) =
    \begin{cases}
    \frac{1}{p} \frac{|x|^p }{t^{p-1}}, & \text{if } t>0, \\
    0,  & \text{if } t=0, x=0, \\
    + \infty  & \text{if } t=0, x \neq 0,  \text{or } t<0. 
    \end{cases}
\end{align*}

\begin{definition}
    If $X$ is a compact measurable space and $(\p, E) \in \mathcal{M}_b(X) \times \mathcal{M}_b(X)^d $, 
\begin{align*}
    \mathcal{B}_p( \p, E) := \sup \left\{ \int a \diff \varrho + \int b \cdot \diff E; (a,b) \in C_b (X,K_p) \right\}.
\end{align*}
\end{definition}
\begin{proposition}
\label{BenamouBrenier}
$\mathcal{B}_p$ is convex and lower semi-continuous on the space $\mathcal{M}_b(X) \times \mathcal{M}_b(X)^d $ for the weak convergence, and moreover we have
\begin{enumerate}
    \item $\mathcal{B}_p(\p, E) \geq 0$,
    \item If $\p$ and $E$ are both absolutely continuous with respect to a positive measure $\lambda$, then $\mathcal{B}_p(\p, E) = \int B_p(\p,E) \diff \lambda$,
    \item $\mathcal{B}_p(\p,E) < \infty $ only if $\p \geq 0 $ and $E \ll \p$,
    \item In that case, we can write $E= v \cdot \p$ and $\mathcal{B}_p(\p,E)= \int \frac{1}{p} |v|^p \diff \varrho$
\end{enumerate}
\end{proposition}

\begin{theorem}
\label{ACCurves}
Let $(\mu_t)_{t \in [0,1]}$ be an absolutely continuous curve in $\mathbb{W}_p(\Omega)$ (for $p > 1$ and $\Omega \subset \mathbb{R}^d$ compact). Then for a.e. $t \in [0,1]$, there exists a vector field $v_t \in L^p(\mu_t; \mathbb{R}^d)$ such that 
\begin{itemize}
    \item the continuity equation $\partial_t \mu_t + \nabla \cdot (v_t \mu_t) = 0$ is satisfied in the weak sense,
    \item for a.e. $t$, we have $\|v_t\|_{L^p(\mu_t)} \leq |\mu'|(t)$ (where $|\mu'|(t)$ denotes the metric derivative at time $t$ of the curve $t \mapsto \mu_t$, w.r.t. the distance $W_p$);
\end{itemize}
Conversely, if $(\mu_t)_{t \in [0,1]}$ is a family of measures in $\mathcal{P}_p(\Omega)$ and for each $t$ we have a vector field $v_t \in L^p(\mu_t; \mathbb{R}^d)$ with $\int_0^1 \|v_t\|_{L^p(\mu_t)} \, dt < +\infty$ solving $\partial_t \mu_t + \nabla \cdot (v_t \mu_t) = 0$, then $(\mu_t)_t$ is absolutely continuous in $W_p(\Omega)$, and for a.e. $t$, we have $|\mu'|(t) \leq \|v_t\|_{L^p(\mu_t)}$.
\end{theorem}

\begin{definition}
    We say that an absolutely continuous curve $(\mu_t)_{t \in [0,1]}$ belongs to $\textnormal{AC}^p(\mathbb{W}_p(\Omega))$ if 
    \begin{align*}
        \int_0 ^T \int_\Omega \left|v_t \right|^p \diff \mu_t \diff t < \infty,
    \end{align*}
    where $(v_t)_{t\in [0,T]}$ is the vector field of minimal $L^p$ norm given by \cref{ACCurves}.
\end{definition}

\begin{definition}
\label{deffunctional}
Let $f: \R_+ \xrightarrow{} \R $ be a convex and l.s.c.\ function, and let $L := \lim_{t\xrightarrow{} + \infty} \frac{f(t)}{t}$. We define
\begin{equation*}
    \mathscr{F}(\mu) := \int_\Omega f(\varrho(x)) \diff x + L \mu^S(\Bar{\Omega}),
\end{equation*}
with $\mu = \varrho \cdot \diff x + \mu^S$, where $\p$ and $\mu^S$ are respectively the absolutely continuous (with respect to the Lebesgue measure) and the singular part of $\mu$. It is a convex and l.s.c. functional for the weak convergence of measures.\\
We will similarly use the notations $\mathscr{H},\mathscr{\Tilde{F}},\mathscr{\Tilde{F}}_1,\mathscr{\Tilde{F}}_2, \mathscr{L}$ to denote the functionals associated with the convex functions $h,\Tilde{f},\Tilde{f_1},\Tilde{f_2}, \ell$ respectively.
\end{definition}

\begin{proposition}
    If $f$ verifies (\textbf{H5}), then the functional $\mathscr{F}$ is geodesically convex in $\mathbb{W}_p$, i.e.\ it is convex along geodesics in $\mathbb{W}_p$.
\end{proposition}

\begin{lemma}
\label{constructsuperlin}
    Let $\varphi : \R_+ \to \R_+$ be a superlinear convex function which is $0$ in a neighborhood of $0$. There exists a smooth strictly convex superlinear function $\Phi$ satisfying (\textbf{H5}) and a constant $C>0$ such that $\Phi \leq C (\varphi+1)$
\end{lemma}
\begin{proof}
We look at the function 
\begin{align*}
    \mathcal{M}\phi (s) := s^d \phi(s^{-d}),
\end{align*}
and define $\Tilde{\Phi}$ to be its lower convex hull. Since $\phi$ is $0$ near $0$, we deduce that $\Tilde{\Phi}$ is eventually constantly $0$, which implies that $\Tilde{\Phi}$ is decreasing, being a convex function. We take 
\begin{align*}
    \Phi (z) = \mathcal{M}^{-1}\Tilde{\Phi} (z) = \Tilde{\Phi}(z^{-\frac{1}{d}})z,
\end{align*}
and find that $\Phi$ is superlinear if $\Tilde{\Phi}(0^+)=+\infty$. Supposing that this is not the case, i.e.\ that there exists some $M >0$ such that $\Tilde{\Phi}\leq M$ near 0, we set
\begin{align*}
    a=\sup_{s >0} \frac{M+1 - \mathcal{M}\phi(s)}{s} < +\infty,
\end{align*}
which is well defined because since $\phi$ is superlinear, $\mathcal{M}\phi(0^+)=+\infty$. We therefore have that
\begin{align*}
    \mathcal{M}\phi(s) \geq M+1 -as,
\end{align*}
so that the lower convex hull of $ \mathcal{M}\phi$, being the supremum of affine functions lower than itself, verifies $\Tilde{\Phi}(0^+) \geq M+1$ when we take $s \to 0$ which yields a contradiction. Up to finding a smooth approximation of $\Tilde{\Phi}$, for example via convolution, we can also assume that $\Phi$ is smooth. It is known (see for example \cite{VillaniO&N} chapter 17) that a superlinear convex function satisfying McCann's condition (\textbf{H5}) is strictly convex away from $0$, therefore up to adding a convex function to $\Tilde{\Phi}$, that is linear near $0$ (which means adding a sublinear function to $\varphi$ near $+\infty$) and that strictly decreases to a constant thereafter (which keeps $\Phi$ bounded near $0$), we can guarantee that $\Phi = \mathcal{M}^{-1} \Tilde{\Phi}$ is strictly convex everywhere, and that $\Phi \leq C (\varphi +1)$ for some constant $C>0$.
\end{proof}

\subsection{Approximating problem}
In this first section we will look at the following variational problem defining the next step in the JKO scheme we will use later :
\begin{align}
\label{jkoproblem}
\p \in \argmin \frac{W^p _p(\p , g)}{p\tau ^{p-1}} + \mathscr{F}(\p).
\end{align}
We will assume that $f$ satisfies (\textbf{H1,H2,H3}).
To gain some regularity for the minimizers of the problem, we will approximate it by adding a small entropy term.
\begin{align}
\label{approxpb}
    \mu_\varepsilon \in \argmin_\mu \frac{W_p ^p (\mu, g)}{p\tau ^{p-1}} + \mathscr{F}_\varepsilon (\mu) 
\end{align}
where $\mathscr{F}_\varepsilon$ is the functional associated with $f_\varepsilon (z) = f(z) + \varepsilon z\log(z)$, i.e.
\begin{align*}
    \mathscr{F}_\varepsilon (\mu) = \mathscr{F} (\mu)+ \begin{cases}
        \varepsilon \int_\Omega \varrho \log (\varrho)\diff x &\textnormal{if } \mu$= $\varrho \diff x \\
        + \infty &\textnormal{otherwise}. 
    \end{cases}
\end{align*}

\begin{lemma}
    The solution of \eqref{approxpb} is given by a density that is bounded from above and away from $0$, and is actually Lipschitz continuous.
\end{lemma}
\begin{proof}
    An easy adaptation of the proof from \cite{BlanchetCarlier}, appendix B.4 shows that the conditions $f_\varepsilon(0^+) = - \infty$ and $f_\varepsilon '(+\infty)=+\infty$ implies that there exists $\delta >0$ such that the solution $\p_\varepsilon$ of \cref{approxpb} satisfies $\delta \leq \p_\varepsilon \leq \frac{1}{\delta}$. Since $f_\varepsilon''$ is bounded from below on compact sets, using the optimality condition (see \cite{OTAM} section 7.2.3 for its derivation), we have
    \begin{equation*}
        \frac{\varphi}{p\tau^{p-1}} + f_\varepsilon ' (\varrho_\varepsilon) = C \textnormal{  a.e.},
    \end{equation*}
    from which we deduce that $\varrho_\varepsilon$ is a Lipschitz function. Indeed, we can write 
    \begin{equation*}
        \p_\varepsilon = (f_\varepsilon ')^{-1} \left( C -  \frac{\varphi}{p\tau^{p-1}} \right),
    \end{equation*}
    and since $\varphi$ is Lipschitz we deduce the result.
\end{proof}

\begin{lemma}[Flow Interchange Technique]
\label{flowinterchange}
    Let $h : \R_+ \to \R$ satisfy (\textbf{H1-H2-H3-H5}) and $\p_\varepsilon$ be the solution of \eqref{approxpb}. Define $\mathscr{H}$ as in \cref{deffunctional}. If $\mathscr{H}(g) < + \infty$ then we have
    \begin{equation*}
        \mathscr{H}(g)  - \mathscr{H}(\p_\varepsilon) \geq \tau \int_\Omega \nabla (L_h(\p_\varepsilon)) \cdot  \left( \nabla f_\varepsilon '(\p_\varepsilon)\right)^{q-1} \diff x \geq 0.
    \end{equation*}
\end{lemma}
\begin{proof}
    Assumption (\textbf{H5}) implies that $\mathscr{H}$ is convex along geodesics, therefore we will look at the geodesic $[0,1]\ni t\mapsto\p_t$ from $\p_\varepsilon$ to $g$. By convexity we have
    \begin{equation*}
        \mathscr{H}(g) - \mathscr{H}(\p_\varepsilon) \geq \lim_{t\to 0} \frac{\mathscr{H}(\p_t)-\mathscr{H}(\p_\varepsilon)}{t}.
    \end{equation*}
    Following \cite{AGS} section 10.4.3, we find that
    \begin{align*}
        \lim_{t\to 0} \frac{\mathscr{H}(\p_t)-\mathscr{H}(\p_\varepsilon)}{t} \geq - \int_\Omega L_h (\p_\varepsilon) \textnormal{Tr}\left(\Tilde{\nabla}(T - id) \right) \diff x,
    \end{align*}
    where $T$ is the optimal transport map from $\p_\varepsilon$ to $g$ and $\Tilde{\nabla}T$ denotes its approximate gradient. Again adapting the proof of Theorem 10.4.5 of \cite{AGS}, up to approximating $T$ with BV functions with nonnegative distributional divergence and using that $L_h(\p_\varepsilon) \geq 0$ is $W^{1,1}(\Omega)$, we have
    {\small
    \begin{align*}
        - \int_\Omega L_h (\p_\varepsilon) \textnormal{Tr}\left(\Tilde{\nabla}(T - id) \right) \diff x &\geq \int_\Omega \nabla \left(L_h(\p_\varepsilon) \right) \cdot \left(T - \textnormal{id}\right) \diff x - \int_{\partial \Omega} L_h (\p_\varepsilon) (T-\textnormal{id} )\cdot n \diff \sigma 
    \end{align*}
    }
    Using the convexity of the domain and the fact that $T$ points inwards $\p_\varepsilon$ a.e. and $L_h(0)=0$, the boundary integral is non positive. Finally, using the optimality conditions for $\p_\varepsilon$ we have
    \begin{align*}
        T- \textnormal{id} = -\left(\nabla \varphi\right)^{q-1}= \tau \left(\nabla f_\varepsilon '(\p_\varepsilon) \right)^{q-1},
    \end{align*}
    and we obtain 
    \begin{align*}
        - \int_\Omega L_h (\p_\varepsilon) \textnormal{Tr}\left(\Tilde{\nabla}(T - id) \right) \diff x &\geq \tau\int_\Omega \nabla (L_h(\p_\varepsilon)) \cdot \nabla \left( f_\varepsilon '(\p_\varepsilon)\right)^{q-1} \diff x \\
        &= \tau \int_\Omega h''(\p_\varepsilon)f_\varepsilon ''(\p_\varepsilon)^{q-1}\left| \nabla \p_\varepsilon \right|^q \diff \p_\varepsilon \geq 0
    \end{align*}
\end{proof}

\begin{lemma}
\label{gammacvg}
    We have $\Gamma$-convergence of the functionals $\mathscr{F}_\varepsilon \xrightarrow[\varepsilon \to 0]{\Gamma} \mathscr{F}$ for the topology of weak convergence of measures.
\end{lemma}

\begin{proof}
    Since $z \mapsto z\log(z)$ is bounded from below, up to adding a constant we can assume that the sequence $\left(\mathscr{F}_\varepsilon \right)_\varepsilon$ is nonincreasing as $\varepsilon \to 0$. This monotonicity ensures that the sequence Gamma converges to the lower semi-continuous relaxation of its pointwise limit (see \cite{braides2002gamma})
    \begin{align*}
        \Tilde{\mathscr{F}} : \mu \mapsto \mathscr{F}(\mu) + \begin{cases} 0
        &\textnormal{if } \mu$= $\varrho \diff x \textnormal{ and }  \int_\Omega \varrho \log (\varrho)\diff x < \infty \\
        + \infty &\textnormal{otherwise}. 
    \end{cases}
    \end{align*}
We now have to show that the l.s.c\ relaxation $\textnormal{sc}^- \Tilde{\mathscr{F}}$ of this functional is indeed $\mathscr{F}$. First we can notice that $\mathscr{F} \leq \Tilde{\mathscr{F}}$ and since $\mathscr{F}$ is l.s.c., we have $\mathscr{F} \leq \textnormal{sc}^- \Tilde{\mathscr{F}}$. To get the opposite inequality, it is enough to find a recovery sequence of probability measures that are given by bounded densities, since bounded densities have finite entropy. We take a probability measure $\mu \in \prob (\Bar{\Omega})$ whose Lebesgue decomposition is $\mu = \p \diff x + \mu^S$ and we can assume that $\mathscr{F}(\mu) < \infty$. We define the set $A_n=\left\{\p > n\right\}$ and
\begin{align*}
    \p_n (x) = \begin{cases}
        \p(x) &\textnormal{ when  } x \in A_n ^c \\
        \fint_{A_n} \p(y) \diff y &\textnormal{ when  } x \in A_n.
    \end{cases}
\end{align*}
We have, using dominated convergence, 
\begin{align*}
    \int_\Omega \left| \p(x) - \p_n(x) \right| \diff x =  \int_{A_n} \left| \p(x) - \fint_{A_n} \p(y) \diff y \right| \diff x  \leq 2 \int_{A_n} \p (x) \diff x \xrightarrow{n \to \infty} 0,
\end{align*}
so that $\p_n$ converges to $\p$ in $L^1$ and $\p_n(\Omega)= \p(\Omega)$.
For the singular part, we approximate $\mu^S$ with linear combinations of Dirac masses supported on $\Omega$, and then take the convolution with a standard mollifier with small enough support so that no mass escapes outside of the domain. We obtain a sequence $\mu_n ^S$ of $L^\infty$ densities such that $\mu_n ^S \to \mu^S$ as measures, and $\mu_n (\Omega) = \mu (\Omega)$. We have therefore constructed a sequence of probability measures $\mu_n = \p_n \diff x + \mu_n ^S$ that converges to $\mu$, and we have, using the convexity of $f$, 
\begin{align*}
    \Tilde{\mathscr{F}}(\mu_n) = \int_\Omega f(\p_n (x) + \mu_n ^S (x)) \diff x \leq \int_\Omega f(\p_n)\diff x + L \int_\Omega \mu_n (x) \diff x.
\end{align*}
For the first term, we use Jensen's inequality to get
\begin{align*}
    \int_\Omega f(\p_n)\diff x &= \int_{A_n ^c} f(\p)\diff x + \int_{A_n} f\left( \fint_{A_n } \p(y) \diff y \right) \diff x \\
    & \leq \int_{A_n ^c} f(\p)\diff x +  \int_{A_n} \left( \fint_{A_n } f(\p(y)) \diff y \right) \diff x = \int_\Omega f(\p) \diff x .
\end{align*}
Remembering that we constructed $\mu_n ^S$ so that the mass of the singular part was conserved, we therefore obtain
\begin{align*}
    \limsup_{n} \Tilde{\mathscr{F}}(\mu_n) \leq \int_\Omega f(\p) \diff x + L \mu^S (\Bar{\Omega}) = \mathscr{F}(\mu),
\end{align*}
so that $\textnormal{sc}^- \Tilde{\mathscr{F}}= \mathscr{F}$
\end{proof}

\begin{lemma}
\label{convergencesolapprox}
    The solution $\mu_\varepsilon$ of the approximating problem \eqref{approxpb} weakly converges as measures to the solution $\mu$ of \eqref{jkoproblem} when $\varepsilon \to 0 $.
\end{lemma}
\begin{proof}
    Using \cref{gammacvg}, and the fact that the Wasserstein distance metrizes the weak convergence of measures, for a given probability measure $g\in \prob(\Bar{\Omega})$, we have
    \begin{align*}
         \mu \mapsto \frac{W_p ^p (\mu, g)}{p\tau ^{p-1}} + \mathscr{F}_\varepsilon (\mu) \xrightarrow[\varepsilon \to 0]{\Gamma} \mu \mapsto   \frac{W_p ^p (\mu, g)}{p\tau ^{p-1}} + \mathscr{F} (\mu).
    \end{align*}
    Since $\Bar{\Omega}$ is compact, $\prob(\Bar{\Omega})$ is compact for the weak convergence of probability measures so that we can assume that the sequence of minimums weakly converge. Because the limit problem has a unique solution $\mu$, we deduce that the sequence of minima $\mu_\varepsilon$ converges to $\mu$. 
\end{proof}

\begin{remark}
    In the sequel, we will always have $g \in L^1(\Omega)$. As a consequence of \cref{constructsuperlin,flowinterchange,convergencesolapprox} one can show that the solution of problem \eqref{jkoproblem} is an absolutely continuous measure $\mu= \p \diff x$ (see the proof of \cref{superlinlemma}). We shall therefore use the usual abuse of notation of denoting by $\p$ any measure $\mu = \p \diff x$ when the context is clear. 
\end{remark}

\section{Energy Dissipation Condition and properties of the limit curve} \par

The results of this section only require assumptions (\textbf{H1,H2,H3}) on the function $f$.
\subsection{Interpolations}

We look at the problem 
\begin{align}
\label{iteratedjko}
\p_{k+1}^\tau \in \argmin \frac{W^p _p(\p , \p_k ^\tau)}{p\tau ^{p-1}} + \mathscr{F}(\p)
\end{align}
Under our hypotheses, this problem has a unique solution and the sequence $(\p _{k})_{k \in \N}$ can be defined recursively.

We can define multiple interpolations which are all useful for our analysis :
\begin{definition}[Piecewise constant interpolation]
We define the piecewise constant interpolation as a piecewise constant curve $\Bar{\p}_t ^\tau$ and an associated velocity $\Bar{v}_t ^\tau$ : for $t \in (k\tau , (k+1) \tau ]$,
\begin{align*}
    &\Bar{\p}_t ^\tau=\p_{k+1} ^\tau, \\
    &\Bar{v}_t ^\tau= v_k ^\tau := \frac{\textnormal{Id}-T_k}{\tau},
\end{align*}
where $T_k$ is the optimal transport map from $\p_k ^\tau$ to $\p_{k+1} ^\tau$. We also define the piecewise constant momentum variable $\Bar{E}_t ^\tau =\Bar{\p}_t ^\tau \Bar{v}_t ^\tau$
\end{definition}

\begin{definition}[Piecewise geodesic interpolation]
We define the piecewise geodesic interpolation as a piecewise geodesic curve $\Tilde{\p}_t ^\tau$ and an associated velocity $\Tilde{v}_t ^\tau$ : for $t \in (k\tau , (k+1) \tau ]$,
\begin{align*}
    &\Tilde{\p}_t ^\tau=\left(\textnormal{Id} -(k\tau -t)v_k ^\tau\right)_\# \p_k^\tau , \\
    &\Tilde{v}_t ^\tau= v_k ^\tau \circ \left(\textnormal{Id} -(k\tau -t)v_k ^\tau\right)^{-1},
\end{align*}
which is constructed to satisfy the continuity equation 
\begin{align*}
    \partial_t \Tilde{\p}_t ^\tau + \nabla \cdot \left( \Tilde{\p}_t ^\tau \Tilde{v}_t ^\tau \right) =0.
\end{align*}
We also define the piecewise geodesic momentum variable $\Tilde{E}_t ^\tau =\Tilde{\p}_t ^\tau \Tilde{v}_t ^\tau$
\end{definition}
For this interpolation, we have for a.e. $t$, 
\begin{align*}
    \lVert \Tilde{v}_t ^\tau \rVert_{L^p (\Tilde{\p} _t ^\tau)}= \left| (\Tilde{\p}  ^\tau)'\right| (t) = \frac{W_p (\p _{k+1}, \p_k ^\tau)}{\tau},
\end{align*}
and therefore the following holds :
\begin{align}
\label{NormeVGeodesique}
    \frac{\tau}{p}  \lVert \Tilde{v}_t ^\tau \rVert_{L^p (\Tilde{\p} _t ^\tau)} = \frac{W^p _ p (\p_{k+1}^\tau, \p_k ^\tau)}{p\tau^{p-1}}. 
\end{align}

\begin{definition}[De Giorgi variational interpolation]
    We define the De Giorgi variational interpolation as a curve $\Hat{\p}_t ^\tau$ : for $t \in (k\tau , (k+1) \tau ]$, we define $s=\frac{t - k\tau}{\tau}\in (0,1]$ and
\begin{align}
\label{DeGiorgiargmin}
    &\Hat{\p}_t ^\tau=  \argmin \frac{W^p _p(\p , \p_k ^\tau)}{p(s \tau) ^{p-1}} + \mathscr{F}(\p).
\end{align}
In particular for $t=(k+1)\tau$ we have $s=1$ and we do retrieve $\Hat{\p}_{k+1} ^\tau = \p_{k+1} ^\tau $, and if $t=k\tau$, $s=0$ and the minimizer has to be $\p =\p_k ^\tau$.
\end{definition}
This interpolation allows us to derive the following precursor to the EDI interpretation of our gradient flow : 
\begin{align}
\label{EDIprecursor}
    \mathscr{F}(\p_k ^\tau) \geq \mathscr{F}(\p_{k+1}^\tau) + \frac{\tau}{q} \int_{k\tau} ^{(k+1)}  \frac{W^p _p (\Hat{\p}_t ^\tau, \p_k ^\tau) }{s (s\tau)^{p-1}} \diff t + \frac{W^p _p (\p_{k+1}^\tau, \p_k ^\tau)}{p \tau^{p-1}}.
\end{align}
Indeed if we define the function 
\begin{align*}
    g : t \in (k\tau, (k+1)\tau ] \mapsto  \min_{\p} \frac{W^p _p(\p , \p_k ^\tau)}{p(s \tau) ^{p-1}} + \mathscr{F}(\p),
\end{align*}
we find that $g$ is a nonincreasing function of $t$, and therefore differentiable almost everywhere. Furthermore, we have a one sided fundamental theorem of analysis :
\begin{align}
\label{FTA}
    \int_{k\tau} ^{(k+1)\tau }  g'(t) \diff t &\geq g((k+1)\tau) - g(k\tau) \\
    & = \frac{W^p _p (\p_{k+1}^\tau, \p_k ^\tau)}{p \tau^{p-1}} + \mathscr{F}(\p_{k+1}^\tau) - \mathscr{F}(\p_k ^\tau). \notag
\end{align}
To find the value of $g'(t)$ when it exists, we can write 
\begin{align*}
    A(t,\p)= \frac{W^p _p(\p , \p_k ^\tau)}{p(s \tau) ^{p-1}} + \mathscr{F}(\p),
\end{align*}
so that $g(t)=\min_{\p} A(t,\p)$. For a point $t_0$ such that $g'(t_0)$ exists we choose $\p_0\in\argmin_\p A(t_0,\p)$; we observe that we have $g(t)\leq A(t,\p_0)$ with equality at $t=t_0$; considering that $t\mapsto A(t,\p_0)$ is differentiable for every $t$, the derivatives of $g$ and of $t\mapsto A(t,\p_0)$ must agree at $t=t_0$, which is a minimum point of the difference. 
Therefore, we find that at differentiability points of $g$ we necessarily have 
\begin{align*}
    g'(t)=\frac{1}{q} \frac{W^p _p (\Hat{\p}_t ^\tau, \p_k ^\tau)}{(s\tau)^{p}},
\end{align*}
which gives \eqref{EDIprecursor} when combined with \cref{FTA}. We can now use the optimality conditions for \eqref{DeGiorgiargmin} : if we denote by $T_s$ and $\varphi_s$ respectively the optimal transport map and a Kantorovich potential between $\hat{\p}_t ^\tau$ and $\p_k ^\tau$, we have
\begin{align*}
    \frac{\nabla \varphi_s}{(s\tau)^{p-1}} = -\nabla f'(\Hat{\p}_s ^\tau)
\end{align*}
which we can replace in the formula for $W_p ^p$ using the definition of the optimal transport map :
\begin{align*}
    T_s- \textnormal{Id} = - \left(\nabla \varphi_s\right)^{q-1},
\end{align*}
and obtain
\begin{align*}
    \frac{1}{q} \frac{W^p _p (\Hat{\p}_t ^\tau, \p_k ^\tau)}{(s\tau)^{p}} = \frac{1}{q} \int_\Omega \frac{\left| x-T_s \right|^p}{(s\tau)^{p}}\diff \hat{\p}_s ^\tau = \frac{1}{q} \int_\Omega \left| \nabla f'(\Hat{\p}_t ^\tau) \right|^q \diff \Hat{\p}_t ^\tau .
\end{align*}
Finally, using \eqref{NormeVGeodesique} in \eqref{EDIprecursor} yields 
\begin{align*}
    \mathscr{F}(\p_k ^\tau) \geq \mathscr{F}(\p_{k+1}^\tau) + \frac{1}{q} \int_{k\tau} ^{(k+1)\tau }\int_\Omega \left| \nabla f'(\Hat{\p}_t ^\tau) \right|^q \diff \Hat{\p}_t ^\tau \diff t + \frac{1}{p} \int_{k\tau} ^{(k+1)\tau } \int_\Omega \left| \Tilde{v}^\tau _ t \right|^p \diff \Tilde{\p}_t ^\tau \diff t  .
\end{align*}
To easily justify the previous computations and later make it easier to pass to the limit we will use the approximation introduced in \eqref{approxpb} to obtain a different (but formally equivalent) form for the slope : 
\begin{align*}
    &\p^\varepsilon=  \argmin \frac{W^p _p(\p , \p_k ^\tau)}{p(s \tau) ^{p-1}} + \mathscr{F}(\p)+ \varepsilon \int_\Omega \p \log \p \diff x.
\end{align*}
We know $\p^\varepsilon$ are Lipschitz and weakly converge to $\Hat{\p}^\tau _t$,  and with the same computations as above we can deduce, using that $\varepsilon\log(\p^\varepsilon)$ is a increasing function of $f'(\p^\varepsilon)$,

\begin{align*}
    \frac{1}{q} \frac{W^p _p (\p^\varepsilon, \p_k ^\tau)}{(s\tau)^{p}} = \frac{1}{q} \int_\Omega \left| \nabla \left( f'(\varrho^\varepsilon) +\varepsilon \log(\p^\varepsilon \right) )\right|^q \diff \varrho^\varepsilon \geq \frac{1}{q} \int_\Omega \left| \frac{ \nabla L_f (\varrho^\varepsilon)}{\p^\varepsilon}  \right|^q \diff \varrho^\varepsilon. 
\end{align*}
Since weak convergence of measures implies the convergence in Wasserstein distance, we have, using the lower semi-continuity of the slope proved in the following lemma :
\begin{align*}
    \frac{1}{q} \frac{W^p _p (\Hat{\p}_t ^\tau, \p_k ^\tau)}{(s\tau)^{p}} \geq \liminf_\varepsilon \frac{1}{q} \int_\Omega \left| \frac{ \nabla L_f (\varrho^\varepsilon)}{\p^\varepsilon}  \right|^q \diff \varrho^\varepsilon \geq  \frac{1}{q} \int_\Omega \left| \frac{ \nabla L_f (\Hat{\p}^\tau _t)}{\Hat{\p}^\tau _t}  \right|^q \diff \Hat{\p}^\tau _t
\end{align*}

\begin{lemma}   
\label{SlopeLSC}
The functional 
\begin{equation*}
    \p \mapsto \frac{1}{q} \int_\Omega \left| \frac{\nabla L_f (\p)}{\p} \right|^q \diff \varrho = \mathcal{B}_q (\varrho, \nabla L_f (\varrho))
\end{equation*}
is lower semi-continuous with respect to the weak $L^1$ convergence.
\end{lemma}

\begin{proof}
Let $\p_n$ be a sequence of probability measures weakly converging in $L^1$ to some probability measure $\p$. We can assume that $\mathcal{B}_q (\varrho_n, \nabla L_f (\varrho_n))$ is bounded and from H\"older's inequality and $\frac{q}{p}=q-1$ we get the following basic estimate :
\begin{align*}
    \int_\Omega \left| \nabla L_f (\p_n) \right| \diff x = \int_\Omega  \frac{\left|\nabla L_f( \p_n)\right|}{\p_n ^{1/p}} \p_n ^{1/p} \diff x &\leq  \left( \int_\Omega  \frac{\left|\nabla L_f( \p_n)\right|^q}{\p_n ^{q/p}} \diff x\right)^\frac{1}{q}\left(\int_\Omega \p_n (x) \diff x \right)^\frac{1}{p} \\
    &\leq C\mathcal{B}_q (\varrho_n, \nabla L_f (\varrho_n))^\frac{1}{q}.
\end{align*}
Using this bound we can therefore deduce that $L_f(\p_n) - C_n$ is bounded in $L^1$ for some constants $C_n$, and using the compactness embedding of $BV$ in $L^1$ we can deduce that up to subsequences, we have, for some $u\in BV(\Omega)$
\begin{align*}
    L_f(\p_n)-C_n &\xrightarrow{} u \textnormal{ almost everywhere and in $L^1$,} \\
    \nabla L_f(\p_n) &\xrightarrow{} \diff u \textnormal{ weakly as measures }.
\end{align*}
This strong convergence implies that $C_n$ is actually bounded. Indeed, if for some further subsequence we had $C_n \to \infty$, this would imply that $L_f(\p_n) \to \infty$ a.e.\ which in turn implies that $\p_n \to \infty$ a.e.\ because $L_f$ is strictly increasing. Similarily, if $C_n \to -\infty$ for some subsequence, we would have $L_f(\p_n) \to -\infty$ a.e.\ and $\p_n \to 0$ a.e. Both situations are impossible in light of \cref{lemmeEgorov} below, so that we can assume that $(C_n)_n$ is bounded. Up to changing $u$ by a constant, we can therefore assume that $L_f(\p_n) \xrightarrow{} u$ a.e. and in $L^1$, which implies that $\p_n \xrightarrow{} (L_f)^{-1}(u)$ a.e. Again using \cref{lemmeEgorov}, the weak convergence $\p_n \to \p $ then gives $u=L_f(\p)$ a.e. Finally, using the lower semicontinuity of $\mathcal{B}_q$, we find that $\mathcal{B}_q (\p,\diff u) < \infty$, so that by using \cref{BenamouBrenier}, we get that $\diff u$ has a density with respect to $\p$ and therefore also with respect to the Lebesgue measure, which is $\nabla L_f (\p)$ and
\begin{align*}
     \frac{1}{q} \int_\Omega \left| \frac{\nabla L_f (\p)}{\p} \right|^q \diff \varrho = \mathcal{B}_q (\varrho, \nabla L_f (\varrho)) &\leq \liminf_n\mathcal{B}_q (\varrho_n, \nabla L_f (\varrho_n)) \\
     &= \liminf_n \frac{1}{q} \int_\Omega \left| \frac{\nabla L_f (\p_n)}{\p_n} \right|^q \diff \varrho_n
\end{align*}
\end{proof}

\begin{lemma}
\label{lemmeEgorov}
    Let $(u_n)_n \in L^1 (\Omega)$ be a sequence of functions that weakly converges to some $u$ in $L^1(\Omega)$, such that $u_n$ converges pointwise to some measurable function $v$ almost everywhere on $A\subset \Omega$, then $u=v$ almost everywhere on $A$. 
\end{lemma}

\begin{proof}
First by Fatou's Lemma, we find 
\begin{equation*}
    \int_A |v| \diff x \leq \|u\|_{L^1 (A)},
\end{equation*}
so that $v$ is finite almost everywhere.
    By Egorov's theorem, for $\varepsilon >0$, there exists some $A_\varepsilon \subset A$ such that $|A\setminus A_\varepsilon | \leq \varepsilon$ and $u_n$ converges to $g$ uniformly on $A_\varepsilon$. By the weak convergence, we have
    \begin{align*}
        \int_{A_\varepsilon} u_n (u-v) \diff x \xrightarrow{} \int_{A_\varepsilon} u (u-v) \diff x.
    \end{align*}
    By the uniform convergence, we find 
    \begin{equation*}
        \int_{A_\varepsilon} u_n (u-v) \diff x \xrightarrow{}\int_{A_\varepsilon} v (u-v) \diff x,
    \end{equation*}
so that we have 
\begin{equation*}
    \int_{A_\varepsilon} (u-v)^2 \diff x =0,
\end{equation*}
and $u=v$ almost everywhere on $A_\varepsilon$. Taking $\varepsilon \to 0$, we find $f=g$ almost everywhere on $A$.
\end{proof}
We can therefore exchange the slope term in the previous partial EDI and obtain
\begin{align*}
    \mathscr{F}(\p_k ^\tau) \geq \mathscr{F}(\p_{k+1}^\tau) + \frac{1}{q} \int_{k\tau} ^{(k+1)\tau }\int_\Omega \left| \frac{\nabla L_f(\Hat{\p}_t ^\tau)}{\Hat{\p}_t ^\tau} \right|^q \diff \Hat{\p}_t ^\tau \diff t + \frac{1}{p} \int_{k\tau} ^{(k+1)\tau } \int_\Omega \left| \Tilde{v}^\tau _ t \right|^p \diff \Tilde{\p}_t ^\tau \diff t .
\end{align*}
Finally, summing on $k$ we get our EDI formulation :
\begin{align}
     \mathscr{F}(\p_0) \geq \mathscr{F}(\p_{T}^\tau) + \frac{1}{q} \int_{0} ^{T }\int_\Omega \left| \frac{\nabla L_f(\Hat{\p}_t ^\tau)}{\Hat{\p}_t ^\tau} \right|^q \diff \Hat{\p}_t ^\tau \diff t + \frac{1}{p} \int_{0} ^{T } \int_\Omega \left| \Tilde{v}^\tau _ t \right|^p \diff \Tilde{\p}_t ^\tau \diff t  ,
\end{align}
which is the desired inequality.
\subsection{Compactness estimates and passing to the limit}
\label{compactness_section}

From our initial JKO scheme, comparing the value of the optimal minimizer $\p_{k+1} ^\tau$ with that of the current step $\p_k ^\tau $ we have
\begin{align*}
    \mathscr{F}(\p_{k+1} ^\tau) + \frac{W^p _p (\p_{k+1} ^\tau, \p_{k} ^\tau)}{p \tau^{p-1}} \leq \mathscr{F}(\p_{k} ^\tau),
\end{align*}
and summing these inequalities yields 
\begin{align*}
    \sum_k  \frac{W^p _p (\p_{k+1} ^\tau, \p_{k} ^\tau)}{p \tau^{p-1}} \leq \mathscr{F}(\p_0) - \inf \mathscr{F} < \infty.
\end{align*}
The above estimate alows us to derive uniform continuity estimates for our interpolations, we skip the computations which are classical and can be found for example in \cite{AGS,OTAM} : if $s<t$ we have
\begin{align*}
    W_p ( \Tilde{\p}_s ^\tau , \Tilde{\p}_t ^\tau) &\leq C (t-s)^\frac{1}{q} \\
    W_p( \Bar{\p}_t ^\tau, \Tilde{\p}_t ^\tau ) &\leq C \tau^\frac{1}{q} \\
    W_p( \Hat{\p}_t ^\tau, \Tilde{\p}_t ^\tau) &\leq C \tau ^\frac{1}{q}.
\end{align*}
Using the Arzel\`a-Ascoli theorem we can deduce that up to a subsequence which we do not relabel, when $\tau \to 0$, all interpolated curves uniformly converge to some $\p \in C^\frac{1}{q} \left([0,T], \mathbb{W}_p(\Omega) \right)$

We now wish to prove compactness for the momentum variable, to pass to the limit in our EDI formulation. To this aim, we can notice that using the Jensen inequality,

\begin{align*}
    \left| \Tilde{E}_t ^\tau \right|(\Omega) &= \int_\Omega \left| \Tilde{v}_t ^\tau \right| \diff \Tilde{\p}_t ^\tau  \leq \left(\int_\Omega \left| \Tilde{v}_t ^\tau \right|^p \diff \Tilde{\p}_t ^\tau \right) ^\frac{1}{p} = \lVert \Tilde{v}_t ^\tau \rVert_{L^p(\Tilde{\p}_t ^\tau)}
\end{align*}
and, adding the time variable,
\begin{align*}
    \left| \Tilde{E} ^\tau \right|\left([0,T] \times \Omega\right) \leq \int_0 ^T \lVert \Tilde{v}_t ^\tau \rVert_{L^p(\Tilde{\p}_t ^\tau)} \diff t \leq T^\frac{1}{q} \left( \int_0 ^T \lVert \Tilde{v}_t ^\tau \rVert_{L^p(\Tilde{\p}_t ^\tau)} ^p \diff t \right) ^\frac{1}{p} \leq C T^\frac{1}{q}
\end{align*}
We can therefore assume that as $\tau \to 0$ and for a subsequence which we do not relabel, $\Tilde{E} ^\tau$ weakly converges as a measure to some measure $E$.

Using the $\mathcal{B}_p$ functional which is lower semi-continuous for the weak convergence of measures, we therefore have that there exists $v\in L^p ([0,T] \times \Omega, \diff \varrho_t \diff t)$ such that $E= \varrho v $ and 

\begin{align*}
    \liminf_\tau \mathcal{B}_p (\Tilde{\p}^\tau, \Tilde{E}^\tau) \geq \mathcal{B}_p (\p,E) = \frac{1}{p} \int_0 ^T \int_\Omega |v_t|^p \diff \varrho_t \diff t 
\end{align*}
Moreover the continuity equation which is linear in $(\p, E)$ readily passes to the limit and we have 
\begin{equation*}
    \partial_t \p + \nabla \cdot \left( \p v \right) =0
\end{equation*}
Up to changing $v$ we can also assume it is the (unique) velocity vector field  of minimal $L^p$ norm.
For the slope term, we only have to use \cref{SlopeLSC} and Fatou's Lemma to get
\begin{align*}
    \liminf_\tau \frac{1}{q} \int_{0} ^{T }\int_\Omega \left| \frac{\nabla L_f(\Hat{\p}_t ^\tau)}{\Hat{\p}_t ^\tau} \right|^q \diff \Hat{\p}_t ^\tau \diff t
    &\geq  \int_0 ^T \liminf_\tau \mathcal{B}_p (\Hat{\p}^\tau _t ,\nabla L_f(\Hat{\p}_t ^\tau)) \diff t \\
    &\geq  \frac{1}{q} \int_0 ^T \int_\Omega \left| \frac{\nabla L_f (\p_t)}{\p_t} \right|^q \diff \varrho_t \diff t.
\end{align*}
Therefore, passing to the limit in \cref{EDIprecursor} gives
\begin{align}
\label{EDI}
     \mathscr{F}(\p_0) \geq \mathscr{F}(\p_{T}) + \frac{1}{q} \int_{0} ^{T }\int_\Omega \left| \frac{\nabla L_f(\p_t )}{\p_t } \right|^q \diff \p_t \diff t + \frac{1}{p} \int_{0} ^{T } \int_\Omega \left| v _ t \right|^p \diff \p_t \diff t  .
\end{align}
\begin{remark}
    The above derivation and estimates are not surprising from the point of view of the metric theory developed in \cite{AGS}. Indeed, the lower semicontinuity from \cref{SlopeLSC} shows that we have 
    \begin{equation*}
       \int_\Omega \left| \frac{\nabla L_f(\varrho)}{\varrho} \right|^q \varrho \, dx \leq |\partial^- \mathscr{F}|^q (\varrho),
    \end{equation*}
    where $|\partial^- \mathscr{F}|$ denotes the relaxed metric slope associated with $\mathscr{F}$, and the well known EDI condition can be obtained by following \cite{AGS}:
    \begin{align*}
         \mathscr{F}(\p_0) \geq \mathscr{F}(\p_{T}) + \frac{1}{q} \int_{0} ^{T } |\partial^- \mathscr{F}|^q (\varrho) \diff t + \frac{1}{p} \int_{0} ^{T } \left| (\p_t)'\right|^p  \diff t  .
    \end{align*}
\end{remark}

\subsection{Estimates on the limit curve} \label{SectionEstimates}

\begin{lemma}
\label{superlinlemma}
    There exists $\Phi: \R_+ \to \R_+$, smooth, strictly convex and superlinear such that 
    \begin{align*}
        \sup_{t\in [0,T]} \int_{\Omega} \Phi(\p_t) \diff x <\infty 
    \end{align*}
\end{lemma}

\begin{proof}
    We know by Dunford-Pettis' theorem that since $\p_0\in L^1 (\Omega)$, there exists some superlinear convex function $\varphi : \R_+ \to \R_+$, which we can assume to be $0$ in a neighborhood of $0$ such that 
    \begin{equation*}
        \int_\Omega \varphi(\p_0) \diff x < \infty 
    \end{equation*}
    We take $\Phi$ to be the one constructed in \cref{constructsuperlin} from $\varphi$. It is smooth and satisfies McCann's condition (\textbf{H5}) so that we can use it in \cref{flowinterchange} : at the $k$-th step of the JKO scheme, we have, for the approximating problem,
    \begin{align*}
        \int_\Omega \Phi(\p_k ^\tau) \diff x \geq \int_\Omega \Phi(\p_\varepsilon) \diff x ,
    \end{align*}
    and taking $\varepsilon \to 0$, we have $\p_\varepsilon \to \p_{k+1} ^\tau$. By lower semi continuity of the functional, we get
    \begin{align*}
         \int_\Omega \Phi(\p_k ^\tau) \diff x \geq \int_\Omega \Phi(\p_{k+1} ^\tau) \diff x,
    \end{align*}
    so that by induction we have, for all $k$,
    \begin{align*}
        \int_\Omega \Phi(\p_k ^\tau) \diff x \leq  \int_\Omega \Phi(\p_0 ) \diff x \leq C \left( \int_\Omega \varphi(\p_0) +1  \diff x \right) <  \infty.
    \end{align*}
    Passing to the limit $\tau \to 0$, and again using the lower semi continuity of the functional, we can deduce that for all $t\geq 0$
    \begin{align*}
        \int_\Omega \Phi(\p_t) \diff x \leq  \int_\Omega \Phi(\p_0 ) \diff x,
    \end{align*}
    and 
    \begin{align*}
         \sup_{t\in [0,T]} \int_{\Omega} \Phi(\p_t) \diff x  \leq \int_{\Omega} \Phi(\p_0) \diff x  <\infty 
    \end{align*}
\end{proof}

\begin{lemma}
\label{integrabilitylemma}
If we have, for $\alpha \geq 1 - \frac{1}{d}$ :
    \begin{align}
    \begin{cases}
        \begin{aligned}
            &\int_\Omega \p_0 \log \p_0 \diff x &< + \infty &\quad\text{if } \alpha = 1, \\[5pt]
            &\int_\Omega \p_0^\alpha \diff x &< + \infty &\quad\text{if } \alpha > 1,
        \end{aligned}
    \end{cases}
\end{align}
then
    \begin{align*}
        \int_{\Omega_T} |\nabla h(\p_t)|^q \diff x \diff t < \infty,
    \end{align*}
 where $h$ is any function satisfying $h'(z)=z^{\frac{\alpha -1}{q}} f''(z)^\frac{1}{p}$, and $\p_t$ is any limit of our $p$-JKO scheme $\cref{jkoscheme}$.
\end{lemma}
\begin{proof}
    We shall use the flow interchange technique, which allows us to gain some first order regularity from the integrability assumption of the initial condition. First we will look at the case where $\alpha \\neq 1$. Starting from the $k$-th step of our JKO scheme, we look at the $\varepsilon$-approximating problem and apply \cref{flowinterchange} with the function $z \mapsto \frac{1}{\alpha -1}z^\alpha$ and get
    \begin{align*}
       \frac{1}{\alpha -1} \int_\Omega (\p^\tau _k) ^\alpha - \frac{1}{\alpha -1} \int_\Omega \p _\varepsilon ^\alpha \geq \tau \int_\Omega \p _\varepsilon ^{\alpha-1}f''(\p_\varepsilon)^{q-1} \left| \nabla \p_\varepsilon \right|^q   \diff x \geq 0,
    \end{align*}
    where $\p_\varepsilon$ is the solution of \cref{approxpb} with $g=\p_k$. We can write $\p_\varepsilon ^{\alpha -1} f''(\p_\varepsilon)^{q-1} \left| \nabla \p_\varepsilon \right|^q = \left|\nabla h(\p_\varepsilon)\right|^q$, and $h$ being strictly increasing, when taking the limit $\varepsilon \to 0$ (and hence $\p_\varepsilon \to \p^\tau _{k+1}$), we can argue as in the proof of \cref{SlopeLSC} to find
    \begin{align*}
        \frac{1}{\alpha -1} \int_\Omega (\p^\tau _k) ^\alpha - \frac{1}{\alpha -1} \int_\Omega (\p^\tau _{k+1}) ^\alpha \geq  \tau \int_\Omega \left|\nabla h(\p^\tau _{k+1}) \right|^q \diff x.
    \end{align*}
    Summing on each step of the JKO scheme, and using Fatou's Lemma when passing to the limit $\tau \to 0$, we can obtain 
    \begin{align*}
       \frac{1}{\alpha -1} \int_\Omega \p_0 ^\alpha \geq \int_{\Omega_T} \left|\nabla h(\p) \right|^q \diff x \diff t + \frac{1}{\alpha -1} \int_\Omega \p _T ^\alpha \diff x.
    \end{align*}

For the case $\alpha=1$, we use the same techniques as above, but applying \cref{flowinterchange} to the function $z \mapsto z \log z$ which satisfies McCann's condition and, in terms of second derivatives, behaves in a similar manner to power functions, so that all computations are the same.
\end{proof}

\section[Derivative along the flow]{Derivative of $\mathscr{F}$ along the flow}
In this section, we will endeavor to compute the derivative of $\mathscr{F}$ along the curve given by $(\p,v)$. Formally, using the continuity equation, we can write 
\begin{align*}
    \frac{\diff }{\diff t} \mathscr{F}(\p_t) = \int_\Omega \frac{\diff }{\diff t} f(\p_t) \diff x = \int_\Omega f'(\p_t) \partial_t \p_t \diff t =\int_\Omega \nabla f'(\p_t) \cdot v_t \diff \p_t,
\end{align*}
which gives
\begin{align*}
    \mathscr{F}(\p_T) - \mathscr{F}(\p_0) = \int _0 ^T \int_\Omega  \nabla f'(\p_t) \cdot v_t \diff \p_t \diff t.
\end{align*}
Now, using $\nabla L_f(\p_t) = \nabla f(\p_t) \p_t$ and the EDI formulation \cref{EDI},
\begin{align*}
    0\geq \int _0 ^T \int_\Omega  \nabla f'(\p_t) \cdot v_t \diff \p_t \diff t+  \frac{1}{q} \int_{0} ^{T }\int_\Omega \left| \nabla f'(\p_t) \right|^q \diff \p_t \diff t + \frac{1}{p} \int_{0} ^{T } \int_\Omega \left| v _ t \right|^p \diff \p_t \diff t,
\end{align*}
which, with Young's inequality gives

\begin{align*}
\begin{cases}
    \partial_t \p_t + \nabla \cdot \left( \p_t v_t\right)  = 0  \\ 
    v_t = - \left(\nabla f'(\p_t)\right)^{q-1}
\end{cases}
\end{align*}
In the rest of this section we shall give a rigorous proof of the above result. 
First we can notice that for our purposes, it is enough to prove that 
\begin{align*}
    \mathscr{F}(\p_T) - \mathscr{F}(\p_0) \geq \int _0 ^T \int_\Omega  \frac{\nabla L_f(\p_t)}{\p_t} \cdot v_t \diff \p_t \diff t.
\end{align*}

\begin{remark}
    The above statement could be proved in the case where $\delta \leq \p  \leq \frac{1}{\delta}$ for some $\delta>0$ (which is the case for example, when $\delta \leq \p_0  \leq \frac{1}{\delta}$ (see \cite{Agueh} or \cite{Otto})) by regularizing the couple $(\p,\p v)$ via convolution. The $L^\infty$ bound corresponds to the worst case $f'' >0$ in our theorem, and we have no need for lower bounds in the following proofs.
\end{remark}

\begin{theorem}[Chain Rule]
\label{ChainRule}
    Let $(\p_t)_{t\in [0,T]} \in \textnormal{AC}^p (\mathbb{W}_p (\Omega))$ with velocity vector field $(v_t)_{t\in [0,T]}$, so that it solves the continuity equation
    \begin{equation*}
        \partial_t \p_t + \nabla \cdot ( \p_t v_t) =0.
    \end{equation*}
    If $\p \in \mathcal{M}([0,T]\times \Omega)$ is absolutely continuous with respect to the Lebesgue measure, $L_f (\p_t) \in W^{1,1}(\Omega)$ for a.e.\ $t \in [0,T]$, and if
    \begin{align*}
        &\int_{\Omega_T} \left| \frac{\nabla L_f(\p_t) }{\p_t} \right|^q \diff \p_t \diff t < \infty,
    \end{align*}
    and, if $f$ does not satisfy McCann's condition,
    \begin{align*}
        &\int_{\Omega_T} |\nabla h(\p)|^q \diff x \diff t < +\infty,
    \end{align*}
    where $h$ is any function satisfying $h'(z)=z^{\frac{\alpha -1}{q}} f''(z)^\frac{1}{p}$, and $\alpha =2- q(1+\frac{1}{d})+ \theta (q-1)$, then
    \begin{align*}
        \mathscr{F} (\p_T) - \mathscr{F} (\p_0) = \int_{\Omega_T} \frac{\nabla L_{f}(\p)}{\p}v \diff \p \diff t.
    \end{align*}
\end{theorem}

\begin{proof}

\textbf{Step 1 : Truncation and decomposition}\\
In this first step, we write $f$ as the difference of two functions $f_1,f_2$ which will satisfy McCann's condition, and for which computations are allowed. To be able to do this we first need to truncate $f$ linearly near $0$ and $+\infty$, and we will pass to the limit at the end to recover the initial function.

    First we linearly truncate $f$ near 0 and infinity and for ease of computations subtract a positive constant such that the truncated function is zero at zero.
    \begin{align*}
        \Tilde{f}(z)=\begin{cases}
            a_0 z &\textnormal{ for } z\leq z_0, \\
            f(z) + b_0 &\textnormal{ for } z_0 \leq z \leq z_1, \\
            a_1 z +b_1 &\textnormal{ for } z_1 \leq z,
        \end{cases}
    \end{align*}
with $a_0 = f'(z_0) \leq a_1 = f'(z_1)$, $b_0 = L_f(z_0)$, $b_1= L_f(z_0) - Lf(z_1)\leq 0$. 
We recall here the definition of the McCann operator $\mathcal{M}$ introduced previously : for $g: \R_+ \to \R$, we write 
\begin{align*}
    \mathcal{M}g (s) =s^d g(s^{-d}),
\end{align*}
and \begin{align*}
    \mathcal{M}^{-1}g(z) = g(z^{-\frac{1}{d}})z
\end{align*}
The McCann function associated to $\Bar{f}$ is
\begin{align*}
    \mathcal{M}\Tilde{f} (s) =
    \begin{cases}
            a_1 + s^d b_1 &\textnormal{ for } s\leq z_1^{-\frac{1}{d}}, \\
           \mathcal{M}f(s) + b_0 s^d &\textnormal{ for } z_1 ^{-\frac{1}{d}} \leq s \leq z_0 ^{-\frac{1}{d}}, \\
            a_0 &\textnormal{ for } z_0 ^{-\frac{1}{d}} \leq s.
        \end{cases}
\end{align*}
We take the second derivative which is 
\begin{align*}
    \left(\mathcal{M}\Tilde{f}\right)'' (s) =
    \begin{cases}
             b_1 (d-1)d s^{d-2} &\textnormal{ for } s\leq z_1^{-\frac{1}{d}}, \\
           \left(\mathcal{M}f\right)'' (s) + b_0 (d-1)d s^{d-2} &\textnormal{ for } z_1 ^{-\frac{1}{d}} \leq s \leq z_0 ^{-\frac{1}{d}}, \\
            0 &\textnormal{ for } z_0 ^{-\frac{1}{d}} \leq s.
        \end{cases}
\end{align*}
We now take $\Tilde{g_2} '' (s)=\left( \left(\mathcal{M}\Tilde{f}\right)'' (s)\right)_-$ and integrate twice so that, choosing the constants $c_0$ and $c_1$ we have
\begin{align*}
    \Tilde{g_2} (s) = \begin{cases}
        -b_1 s^d + c_1 s &\textnormal{ for } s\leq z_1^{-\frac{1}{d}}, \\
           \Theta(s) &\textnormal{ for } z_1 ^{-\frac{1}{d}} \leq s \leq z_0 ^{-\frac{1}{d}}, \\
            c_0 &\textnormal{ for } z_0 ^{-\frac{1}{d}} \leq s,
    \end{cases}
\end{align*}
where $\Theta$ is the double primitive of $\Tilde{g_2} (s)$ with suitable integration constants so that $\Tilde{g}_2$ is $C^2$ (except at $z_1 ^{-\frac{1}{d}}$ and $ z_0 ^{-\frac{1}{d}}$). $g_2$ is convex by construction, and also decreasing since it is constant at infinity.
And finally we can define
\begin{align*}
    \Tilde{f}_2 (z) = \mathcal{M}^{-1}\Tilde{g}_2  (z) =
    \begin{cases}
            c_0 z &\textnormal{ for } z\leq z_0, \\
            \mathcal{M}^{-1} \Theta (z) &\textnormal{ for } z_0 \leq z \leq z_1, \\
            c_1 z^{1- \frac{1}{d}} - b_1 &\textnormal{ for } z_1 \leq z,
        \end{cases}
\end{align*}
$\Tilde{f}_2$ is a $C^2$ (except at $z_0$ and $z_1$) convex function that is linear near the origin and that satisfies the McCann condition and
\begin{align*}
    \Tilde{f}_2 ''(z) &\leq C z^{-(1+\frac{1}{d})} \text{  when  } z \geq z_0.
\end{align*}
Now we set 
\begin{align*}
    \Tilde{f}_1 =\Tilde{f} + \Tilde{f}_2,
\end{align*}
which, by construction also satisfies McCann's condition ($\mathcal{M}\Tilde{f}_1$ is convex and constant for large enough values). 

\textbf{Step 2 : Getting regularity for the slope of $f_2$}\\
In this second step, we use the integrability assumption on $\nabla h(\p)$ to gain the necessary integrability on $\nabla L_{\Tilde{f}_2} (\p)$.

To obtain our needed regularity we will use a finite differences trick, indeed we can write, for $z \geq z_0$,
\begin{align*}
    \left(L_{\Tilde{f}_2} \right)'(z) = z\Tilde{f}_2 ''(z) \leq Cz^{-\frac{1}{d}} \leq C z ^{\frac{q-1}{q}} h'(z) z^{-\frac{1}{d} + \frac{q-1}{q}\theta - \frac{\alpha -1}{q} - \frac{q-1}{q}} = C z ^{\frac{q-1}{q}} h'(z),
\end{align*}
where we used the hypothesis $f''(r)\geq C r^{-\theta}$ and our condition $\alpha =2- q(1+\frac{1}{d})+ \theta (q-1)$ so that $-\frac{1}{d} + \frac{q-1}{q}\theta - \frac{\alpha -1}{q} - \frac{q-1}{q}=0$. Since $\Tilde{f}_2 '' (z)=0$ when $z<z_0$, the above inequality is also trivially true in this case.
We can therefore write, for $s>t$,
\begin{align*}
    L_{\Tilde{f}_2} (s) - L_{\Tilde{f}_2} (t) = \int_t ^s \left(L_{\Tilde{f}_2}\right)' (z) \diff z \leq C \int_t ^s z^{\frac{q-1}{q}} h'(z) \diff z 
\end{align*}
and 
\begin{align*}
    \left| L_{\Tilde{f}_2} (s) - L_{\Tilde{f}_2} (t) \right| \leq C \max (s,t)^\frac{q-1}{q} \left| h(s) - h(t) \right|
\end{align*}
Taking a test function $\varphi \in C_c ^\infty (\Omega)$ and $\delta \in \R^d$, we have

\begin{align*}
    \Big|\int_\Omega L_{\Tilde{f}_2} (\p(x)) & \nabla \varphi (x) \cdot \delta \diff x \Big|\\
    =& \left|\lim_{t\to 0} \int_\Omega L_{\Tilde{f}_2} (\p(x)) \frac{\varphi(x -t\delta) - \varphi(x)}{t} \diff x \right| \\
    =&\left| \lim_{t\to 0} \int_\Omega  \frac{L_{\Tilde{f}_2} (\p(x+t\delta)) - L_{\Tilde{f}_2} (\p(x))}{t} \varphi(x) \diff x \right|\\
     \leq& \limsup_{t\to 0} \int_\Omega \max ( \p(x+t\delta), \p(x))^{\frac{q-1}{q}} \frac{\left|h(\p(x+t\delta)-h(\p(x)) \right|}{t} \left|\varphi(x) \right| \diff x \\
    =& \int_\Omega \p(x)^{\frac{q-1}{q}}\left|\nabla h(\p) (x) \cdot \delta  \right| \left| \varphi(x) \right|\diff x,
\end{align*}

where we used the strong convergence $\p(\cdot +t \delta) \to \p$ in $L^1$ and $\frac{\left|h(\p(\cdot+t\delta)-h(\p) \right|}{t} \to \left|\nabla h(\p) \cdot \delta  \right|$ in $L^q$. This proves that $L_{\Tilde{f}_2} (\p)$ is a $W^{1,1}$ function such that $\left| \nabla L_{\Tilde{f}_2} (\p) \right| \leq \p^\frac{q-1}{q}\left| \nabla h(\p) \right|$, so that
\begin{align*}
    \int_{\Omega_T} \frac{\left| \nabla L_{\Tilde{f}_2} (\p)\right|^q}{\p^{q-1}} \diff x \diff t \leq \int_{\Omega_T} \left|\nabla h (\p) \right|^q \diff x \diff t < \infty 
\end{align*}

\textbf{Step 3 : Proving the chain rule}\\
In this step, we will prove the chain rule for the truncated function $\Tilde{f}$. 
\begin{lemma}
\label{derivationformula}
Let $(\p_t)_{t\in [0,T]} \in \textnormal{AC}^p (\mathbb{W}_p (\Omega))$ with velocity vector field $(v_t)_{t\in [0,T]}$, so that it solves the continuity equation
    \begin{equation*}
        \partial_t \p_t + \nabla \cdot ( \p_t v_t) =0,
    \end{equation*}
    and assume that $\p \in \mathcal{M}([0,T]\times \Omega)$ is absolutely continuous with respect to the Lebesgue measure.\\
    Let $\ell :\R_+ \to \R$ be a $C^1$ function satisfying (\textbf{H1-H2-H3-H5}), such that $L_\ell(\p) \in L^1 _{t,x}$, and
    \begin{align*}
        \int_{\Omega_T} \frac{\left|\nabla L_\ell(\p)\right|^q}{\p^{q-1}} \diff x \diff t <+\infty,
    \end{align*}
    then
    \begin{align*}
        \mathscr{L}(\p_T) - \mathscr{L}(\p_0) = \int_{\Omega_T} \frac{\nabla L_\ell(\p)}{\p} \cdot v \diff \p \diff t.
    \end{align*}
\end{lemma}
The proof of this Lemma can be obtained by combining the results of Theorems 10.3.18 and 10.4.6 from \cite{AGS}, but for the sake of completeness and to give a proof that does not use the heavy formalism developed throughout a whole book, we give an independent, even if somewhat technical, proof of this Lemma in \cref{appendix}.

We use \cref{derivationformula} with the functions $h=\Tilde{f}_1$, and $h=\Tilde{f}_2$ to obtain

\begin{corollary}
    With our previous notations, we have the following :
    \begin{align*}
        \Tilde{\mathscr{F}}_1 (\p_T) -  \Tilde{\mathscr{F}}_1 (\p_0) = \int_{\Omega_T} \frac{\nabla L_{\Tilde{f}_1}(\p)}{\p} \cdot v \diff \p \diff t, \\
        \Tilde{\mathscr{F}}_2 (\p_T) -  \Tilde{\mathscr{F}}_2 (\p_0) = \int_{\Omega_T} \frac{\nabla L_{\Tilde{f}_2}(\p)}{\p} \cdot v \diff \p \diff t,
    \end{align*}
    so that
    \begin{align*}
        \Tilde{\mathscr{F}} (\p_T) -  \Tilde{\mathscr{F}} (\p_0) =
        \int_{\Omega_T} \frac{\nabla L_{\Tilde{f}}(\p)}{\p}\cdot v \diff \p \diff t
    \end{align*}
\end{corollary}

\textbf{Step 4 : Passing the truncation to the limit}\\
We will now prove that we can pass our truncation process to the limit and still keep our inequality. As a reminder, we defined our linear truncation $\Tilde{f}_n$ as 
 \begin{align*}
        \Tilde{f}_n (z)=c^n + \begin{cases}
            a_0 ^n z  &\textnormal{ for } z\leq z_0 ^n, \\
            f(z) + b_0 ^n &\textnormal{ for } z_0^n  \leq z \leq z_1 ^n , \\
            a_1 ^n  z +b_1 ^n   &\textnormal{ for } z_1 ^n \leq z,
        \end{cases}
    \end{align*}
with $a_0 ^n = f'(z_0 ^n ) \leq a_1 ^n  = f'(z_1 ^n)$, $b_0 ^n  = L_f(z_0 ^n )$, $b_1 ^n = L_f(z_0 ^n ) - Lf(z_1 ^n )\leq 0$ and $c^n =f(z_0 ^n) - a_0 ^n z_0 ^n$. This formula just amounts to setting $\Tilde{f}_n ''(z)=0$ when $z\leq z_0 ^n$ and $z\geq z_1 ^n$. Since $f$ is convex, we have $\Tilde{f}_n \leq f$, so that by the monotone convergence theorem, we directly have
\begin{align*}
    \Tilde{\mathscr{F}}_n (\p) \to \mathscr{F}(\p)
\end{align*}
for any density $\p \in L^1 (\Omega)$. \\
For the slope term, since $L_h ' (z)=zh''(z)$, we have $L_{\Tilde{f}_n} (\p) = \left( \alpha^n \wedge (\beta^n \vee L_f (\p)) \right)$ for some constants $\alpha^n$ and $\beta^n$, so that $L_{\Tilde{f}_n}(\p)$ is a Lipschitz function of $L_f (\p)$ and we can write 
\begin{align*}
    \nabla L_{\Tilde{f}_n}({\p}) = \nabla L_{f}(\p) \mathbbm{1}_{\{\alpha^n \leq \p \leq \beta^n \}}
\end{align*}
Using dominated convergence, we can therefore conclude that
\begin{align}
    \label{finalineq}
    \mathscr{F} (\p_T) - \mathscr{F} (\p_0) =\int_{\Omega_T} \frac{\nabla L_{f}(\p)}{\p}\cdot v \diff \p \diff t
\end{align}
In the case when $f$ already satisfies McCann's condition (\textbf{H5}), one can skip steps 1,2 and 4 and directly use step 3 with the untruncated function $f$ itself.
\end{proof}

We are now ready to give the proof of \cref{maintheorem}.

\begin{proof}
    We will prove that the curve $(\p_t)_{t\in [0,T]}$ we constructed in Section 3 is the solution of the PDE \eqref{PDE}. We know that there exists a velocity field $(v_t)_{t\in [0,T]}$ such that we have
    \begin{align*}
        \partial_t \p_t + \nabla \cdot \left( \p_t v_t\right) =0,
    \end{align*}
    and the EDI condition \eqref{EDI}
    \begin{align*}
        \mathscr{F}(\p_0) \geq \mathscr{F}(\p_{T}) + \frac{1}{q} \int_{0} ^{T }\int_\Omega \left| \frac{\nabla L_f(\p_t )}{\p_t } \right|^q \diff \p_t \diff t + \frac{1}{p} \int_{0} ^{T } \int_\Omega \left| v _ t \right|^p \diff \p_t \diff t  .
    \end{align*}
Using \cref{superlinlemma,integrabilitylemma}, we can apply \cref{ChainRule} to obtain 
\begin{align*}
        \mathscr{F} (\p_T) - \mathscr{F} (\p_0) = \int_{\Omega_T} \frac{\nabla L_{f}(\p)}{\p}v \diff \p \diff t,
    \end{align*}
so that we have
\begin{align*}
    0 \geq \int_{\Omega_T}  \frac{\nabla L_f (\p)}{\p} \cdot  v + \frac{1}{q} \left|\frac{\nabla L_f (\p) }{\p} \right|^q+ \frac{1}{p}|v|^p \diff \p \diff t.
\end{align*}
According to Young's inequality, this implies that
\begin{align*}
    v=-\left(\frac{\nabla L_f (\p)}{\p} \right)^{q-1} \quad \quad \quad \p \textnormal{ a.e. for a.e. } t.
\end{align*}

\end{proof}

\section{BV Estimates for the solution}

In this section we give some BV estimates on the limit curve obtained by our $p$-JKO scheme, both our results are obtained using the flow interchange technique, with suitable assumptions on the initial condition and the function $f$.
The following lemma, proved in \cite{FiveGradsIneq}  (see also \cite{DiMMurRad}) as a generalization of the results of \cite{BVEstimates}, is useful to prove that BV bounds are conserved along out $p$-JKO scheme.

\begin{lemma}
\label{5grads}
Let $\Omega \subset \R^d$ be bounded and convex with non-empty interior, $\varrho,g \in W^{1,1}(\Omega)$ be two probability densities, $h\in C^1(\R^d)$ a radially symmetric strictly convex function and $H \in C^1(\R^d \backslash \{0\})$ be a radially symmetric convex function, then the following inequality holds 
\begin{equation}
    \int_{\Omega} \big( \nabla \varrho \cdot \nabla H ( \nabla \varphi) + \nabla g \cdot \nabla H (\nabla \psi) \big) \diff x \geq 0,
\end{equation}
where $(\varphi,\psi)$ is a choice of Kantorovich potentials for the optimal transport problem between $\varrho$ and $g$ for the transport cost given by $c(x,y)=h(x-y)$, with the convention that $\nabla H (0) = 0$.
\end{lemma}

\begin{lemma}
    Let $g\in \prob (\Omega) \cap BV(\Omega)$, and denote by $\p$ the solution of problem \cref{jkoproblem}. Then $\p \in BV(\Omega)$ and $\|\p\|_{BV(\Omega)} \leq \|g\|_{BV(\Omega)}$
\end{lemma}
\begin{proof}
    For the case $p=2$, the proof is done in \cite{FokkerPlanckLp} using the five gradients inequality for $p=2$. For the sake of completeness, we give below the same proof adapted to the general case, using the generalized five gradients inequality \cref{5grads}. As usual, we start by looking at the approximating problem \cref{approxpb}, and assume that $g\in W^{1,1} (\Omega)$. Applying \cref{5grads}, with the function $H(z)=|z|$, we obtain 
    \begin{align*}
        \int_\Omega \left( \nabla \p_\varepsilon \cdot \frac{\nabla \varphi}{|\nabla \varphi|} + \nabla g \cdot \frac{\nabla \psi}{| \nabla \psi|} \right) \diff x \geq 0
    \end{align*}
    Using the optimality condition $f'' _ \varepsilon ( \p_\varepsilon) \nabla \p_\varepsilon = - \frac{\nabla \varphi}{\tau^{p-1}}$, we know that $\nabla \p_\varepsilon$ and $\nabla \varphi$ point in opposite directions, so that we get
    \begin{align*}
        \int_\Omega | \nabla \p_\varepsilon | \diff x \leq \int_\Omega  \nabla g \cdot \frac{\nabla \psi}{| \nabla \psi|} \diff x \leq \int_\Omega |\nabla g | \diff x.
    \end{align*}
    Passing to the limit $\varepsilon \to 0$, we obtain $\|\p \|_{BV} \leq \|g\|_{W^{1,1}}$, and we pass to $g \in BV(\Omega)$ by approximating $g$ by $g_n \in W^{1,1}$ with converging $BV$ norm. This yields uniform convergence of the functional $\p \mapsto W_p (\p,g_n)$ to $\p \mapsto W_p (\p, g)$ and therefore $\Gamma$-convergence of the functional in problem \eqref{jkoproblem}, so that the corresponding minimizers $\p_n$ converge to the minimizer of the limit problem $\p$. Using the lower semi continuity of the $BV$ norm, we obtain the claimed inequality $\|\p\|_{BV(\Omega)} \leq \|g\|_{BV(\Omega)}$
\end{proof}
A direct application of the above result immediately gives the following corollaries :
\begin{corollary}
\label{bvdecreaseJKO}
    The recursive sequence $(\p_k ^\tau)$ obtained from the $p$-JKO scheme \eqref{iteratedjko} satisfies\linebreak $\|\p_{k+1}\|_{BV} ^\tau \leq \| \p_k ^\tau\| _{BV}$.
\end{corollary}

\begin{corollary}
\label{bvdecreaseJKO2}
    If $\p_0 \in BV(\Omega)$, then $\p_t \in BV(\Omega)$ for all $t\geq 0$ and $\|\p_t\|_{BV} \leq \|\p_0 \|_{BV}$. 
\end{corollary}

\begin{remark}
    If we know that the solution to \eqref{PDE} is unique, when given some $t\geq 0$, we can use \cref{bvdecreaseJKO2} to prove that for all $s\geq 0$, we have $\|\p_{t+s}\|_{BV} \leq \|\p_t \|_{BV}$. Indeed, we can look at the same $p$-JKO scheme with $\p_t$ as an initial condition, and obtain a solution $(\Bar{\p}_s)_s$ such that $\|\Bar{\p}_{s}\|_{BV} \leq \|\p_t \|_{BV}$. Uniqueness gives $\Bar{\p}_{s}=\p_{t+s}$ and the result. 
\end{remark}

We now prove an instantaneous $BV$ regularization property for the limit curve constructed by our $p$-JKO scheme. Notice that in the following theorem, we assume that $f''(z) \geq Cz^{-\theta}$ not only for large $z$, but for $z$ near $0$ as well. 

\begin{theorem}
    Assume $f''(z) \geq Cz^{-\theta}$ for all $z\in \R_+$, $\theta \geq - \frac{p-1}{d}$, and let $\beta =\max\left(1- \frac{1}{d};\frac{\theta}{p}+\frac{1}{q} \right)$. Assume moreover that $\p_0\in L^\beta$ whenever $\beta>1$, or $\int_\Omega \p_0 \log \p_0 \diff x < + \infty $ if $ \beta = 1$, then there exists some $\Tilde{C}>0$ such that for all $t>0$, we have $\|\p_t \|_{BV} \leq \Tilde{C} t^{-\frac{1}{q}}$
\end{theorem}

\begin{proof}
    We will once again use the flow interchange technique, starting from the $k$-th step of our $p$-JKO scheme, we look at the $\varepsilon$ approximating problem and apply \cref{flowinterchange} with the function $z \mapsto \frac{1}{\beta -1} z^\beta$ (or $z \mapsto z\log(z)$ in the case $\beta=1$ with similar computations) to get
    \begin{align*}
        \frac{1}{\beta -1} \int_\Omega (\p_k ^\tau)^\beta \diff x - \frac{1}{\beta -1} \int_\Omega \p^\beta _\varepsilon \diff x &\geq \tau \beta \int_\Omega \p^{\beta -1} _\varepsilon f''(\p_\varepsilon)^{q-1} |\nabla \p_\varepsilon |^q \diff x \\
        & \geq C \beta \tau \int_\Omega \p_\varepsilon ^{\beta - 1 - \theta (q-1)} |\nabla \p_\varepsilon |^q \diff x.
    \end{align*}
    Now, using H\"older's inequality we obtain 
    \begin{align*}
        \int_\Omega |\nabla \p_\varepsilon | \diff x = \int_\Omega \p_\varepsilon^{-\gamma}\p_\varepsilon ^\gamma |\nabla \p_\varepsilon | \diff x \leq \left( \int_\Omega \p_\varepsilon ^{\gamma p} \diff x \right)^{\frac{1}{p}} \left( \int_\Omega \p_\varepsilon ^{-\gamma q}  |\nabla \p_\varepsilon |^q\diff x \right)^{\frac{1}{q}},
    \end{align*}
    where we chose $\gamma$ to statisfy $-\gamma q =\beta -1 -\theta (q-1)$. In order for $ \int_\Omega \p_\varepsilon ^{\gamma p} \diff x$ to be bounded, we want $\gamma$ to be such that $0 \leq \gamma p \leq \beta$, which is equivalent to $\frac{\theta}{p}+\frac{1}{q} \leq \beta \leq \left( \frac{\theta}{p}+\frac{1}{q} \right) q$ . To apply the flow interchange technique, we also need $\beta \geq 1 - \frac{1}{d}$, so that the condition $\theta \geq -\frac{p-1}{d}$, which implies $q \left( \frac{\theta}{p}+\frac{1}{q} \right) \geq 1 - \frac{1}{d}$, allows us to take $\beta =\max\left(1- \frac{1}{d};\frac{\theta}{p}+\frac{1}{q} \right)$.  We can therefore write
    \begin{align*}
        \tau \left( \int_\Omega |\nabla \p_\varepsilon | \diff x \right)^q \leq \left( \int_\Omega (\p_k ^\tau)^{\gamma p} \diff x \right)^\frac{q}{p} \left(\frac{1}{\beta -1} \int_\Omega (\p_k ^\tau)^\beta \diff x - \frac{1}{\beta -1} \int_\Omega \p^\beta _\varepsilon \diff x\right),
    \end{align*}
    and taking the limit $\varepsilon \to 0$ and summing over steps of the scheme gives
    \begin{align*}
        \sum_k \tau \left( \int_\Omega |\nabla \p_k ^\tau | \diff x  \right)^q \leq \Tilde{C}.
    \end{align*}
    Using \cref{bvdecreaseJKO} we can obtain
    \begin{align*}
        \int_\Omega |\nabla \Bar{\p}_t ^\tau | \leq \Tilde{C} t^{-\frac{1}{q}},
    \end{align*}
    and passing to the limit $\tau \to 0$ we have $\|\p_t\|_{BV(\Omega)} \leq \Tilde{C} t^{-\frac{1}{q}}$ for all $t>0$.
\end{proof}

\appendix

\section{Proof of the chain rule}
\label{appendix}
This section is devoted to the proof of \cref{derivationformula}, the idea is to discretize the curve $(\p_t)_t$ and approximate it with its geodesic interpolation, for which we can use the geodesic convexity of $\ell$ to obtain the discrete chain rule. We then pass to the limit and recover the chain rule for the original curve.
\begin{proof}
    Since $\p \in \mathcal{M}([0,T]\times \Omega)$ is absolutely continuous with respect to the Lebesgue measure, the Dunford-Pettis theorem yields the existence of some superlinear function $\Phi : \R_+ \to \R_+$, such that 
    \begin{align*}
        \int_0 ^T \int_\Omega \Phi(\p_t) \diff x \diff t < \infty,
    \end{align*}
    which we can assume to be smooth and strictly convex. We use \cref{sommes} to find a sequence of partitions $0=t_0 ^N<t_1 ^N< \dots <t_{N+1} ^N =T$ of $[0,T]$ with $\displaystyle \sup_{i} t_{i+1} ^N - t_i ^N \xrightarrow{N \to \infty} 0$, such that
    \begin{align*}
        \int_\Omega L_\ell(\p_{t_i ^N}) \diff x < \infty,
    \end{align*}
    and
\begin{align*}
    \sum_{k=1} ^N (t_{i+1} ^N - t_i ^N)A(t_i ^N) \to \int_0 ^T A(t) \diff t,
\end{align*}
where $A(t)=\displaystyle\int_{\Omega} \left|\frac{\nabla L_\ell(\p_t)}{\p_t} \right|^q \diff \p_t + \displaystyle\int_{\Omega} \Phi(\p_t)\diff x$. Like before, we denote by $\Bar{\p} ^N$ the piecewise constant interpolation of $(\p_{t_i ^N})_i$, defined by $\Bar{\p}_t ^N = \p_{t_i ^N}$ if $t\in [t_i ^N,t_{i+1} ^N[$ for $i=1,\dots,N$, and $\Bar{\p}_t ^N=0$ for $t\in [0,t_1[$ so that we have
\begin{multline}
    \label{cvgsome}
          \int_{\Omega_T} \left|\frac{\nabla L_\ell(\Bar{\p}_t ^N)}{\Bar{\p}_t ^N} \right|^q \diff \Bar{\p}_t ^N \diff t + \int_{\Omega_T} \Phi(\Bar{\p}_t ^N) \diff x \diff t \\\xrightarrow{N\to +\infty}   \int_{\Omega_T} \left|\frac{\nabla L_\ell(\p_t)}{\p_t} \right|^q \diff \p_t \diff t + \int_{\Omega_T} \Phi(\p_t) \diff x \diff t < \infty .
    \end{multline}
Using the same argument as in the proof of \cref{flowinterchange}, the geodesic convexity of $\mathscr{L}$, and the fact that $L_\ell(\p)$ is $W^{1,1}$ allows us to write 
\begin{align*}
    \mathscr{L}(\p_{t_{i+1} ^N}) - \mathscr{L}(\p_{t_i ^N}) &\geq  \int_\Omega \frac{\nabla L_{\ell} (\p_{t_i ^N})}{\p_{t_i ^N}}(T_i  - \textnormal{id} )\diff \p_{t_i ^N}\\
    &=\int_{t_i ^N} ^{t_{i+1} ^N} \int_\Omega \frac{\nabla L_\ell(\Bar{\p}_t ^N)}{\Bar{\p}_t ^N} \Bar{v}^N _t \diff \Bar{\p}_t ^N \diff t,
\end{align*}
where $T_i$ is the optimal transport map from $\p_{t_i ^N}$ to $\p_{t_{i+1} ^N}$ and $\Bar{v}_t ^N = \frac{T_i - \textnormal{id} }{t_{i+1} ^N-t_i ^N}$. 
Summing everything gives
\begin{align}
\label{trucpasserlimite}
    \mathscr{L}(\p_T)-\mathscr{L}(\p_{t_1}) \geq \int_{\Omega_T} \frac{\nabla L_\ell(\Bar{\p}_t ^N)}{\Bar{\p}_t ^N} \Bar{v}_t ^N \diff \Bar{\p}_t ^N \diff t.
\end{align}
We will now prove that this inequality passes to the limit. Since $\p \in C^0([0,T], \mathbb{W}_p)$, we have that $\Bar{\p}_t ^N \to \p_t$ weakly in $L^1$ for any fixed $t\in ]0,T]$. From \cref{deffunctional,SlopeLSC} and using Fatou's Lemma, we therefore obtain that
\begin{align*}
     \liminf \int_{\Omega_T} \left|\frac{\nabla L_\ell(\Bar{\p}_t ^N)}{\Bar{\p}_t ^N} \right|^q \diff \Bar{\p}_t ^N \diff t &\geq  \int_{\Omega_T} \left|\frac{\nabla L_\ell(\p)}{\p} \right|^q \diff \p \diff t  \\
     \liminf \int_{\Omega_T} \Phi(\Bar{\p}_t ^N) \diff x \diff t &\geq \int_{\Omega_T} \Phi(\p_t) \diff x \diff t.
\end{align*}
Together with \eqref{cvgsome}, this implies that 
\begin{align*}
      \int_{\Omega_T} \left|\frac{\nabla L_\ell(\Bar{\p}_t ^N)}{\Bar{\p}_t ^N} \right|^q \diff \Bar{\p}_t ^N \diff t &\to  \int_{\Omega_T} \left|\frac{\nabla L_\ell(\p_t)}{\p_t} \right|^q \diff \p_t \diff t  \\
     \int_{\Omega_T} \Phi(\Bar{\p}_t ^N) \diff x \diff t &\to  \int_{\Omega_T} \Phi(\p_t) \diff x \diff t.
\end{align*}
We will now use a classical trick of convex functions to obtain the almost everywhere convergence of $\Bar{\p}_t$ (see \cite{evans1990weak}). We define $\omega$ to be the convexity gap of $\Phi$, which is to say
\begin{align*}
    \omega(z,x)= \Phi(z) - \Phi(x) - \Phi'(x)(z-x) \geq 0.
\end{align*}
By the strict convexity of $\Phi$, $\omega(z,x)=0$ if and only if $z=x$, and if $(z_n)_n$ is a sequence such that $\omega(z_n,x)\to 0$, then $z_n \to x$. Let $A_M = \left\{\p \leq M\right\}$, we have:
\begin{multline*}
    \int_{\Omega_T} \Phi(\Bar{\p}_t ^N) \diff x \diff t \geq \int_{A_M} \Phi(\p_t) \diff x \diff t + \int_{A_M} \Phi'(\p_t)(\Bar{\p}_t ^N - \p_t)\diff x \diff t \\
    + \int_{A_M} \omega(\Bar{\p}_t ^N, \p_t) \diff x \diff t.
\end{multline*}
We know that $\Bar{\p}^N$ weakly converges in $L^1 _ {t,x}$ to $\p$, and on $A_M$, $ \Phi'(\p)$ is bounded, so that the second integral vanishes in the limit. Taking the $\limsup$ we therefore get
\begin{align*}
    \int_{A_M ^c} \Phi(\p_t) \diff x \diff t \geq \limsup \int_{A_M} \omega(\Bar{\p}_t ^N,\p_t) \diff x \diff t.
\end{align*}
Notice that the left hand side is decreasing in $M$, while the right hand side is increasing in $M$, therefore, taking the limit $M\to +\infty$, we have
\begin{align*}
    0=\lim \int_{\Omega_T} \omega(\Bar{\p}_t ^N,\p_t) \diff x \diff t
\end{align*}
Up to extraction, we can therefore assume that $\omega(\Bar{\p},\p) \to 0$ almost everywhere, which, as stated above entails that $\Bar{\p} \to \p$ almost everywhere. 
Again \cref{cvgsome} implies that $\int_{\Omega_T} \left| \nabla L_\ell (\Bar{\p}_t ^N) \right| \diff x \diff t$ is bounded, so that up to extraction we can assume $(\nabla L_\ell (\Bar{\p}^N))$ converges to some measure. Using the strong convergence of $\Bar{\p}$, one can identify this limit as $\nabla L_\ell(\p)$. We also know that $   \int_{\Omega_T} \left|\frac{\nabla L_\ell(\Bar{\p}_t ^N )}{\Bar{\p}_t ^N} \right|^q \diff \Bar{\p}_t ^N \diff t$ is bounded, so that up to extraction, $(\nabla L_\ell (\Bar{\p}^N) (\Bar{\p}^N)^{-\frac{1}{p}})$ weakly converges in $L^q _{t,x}$ to some $\xi$. Multiplying by $(\Bar{\p}^N)^\frac{1}{p}$, which strongly converges to $\p^\frac{1}{p}$ in $L^p _ {t,x}$, we obtain $\nabla L_\ell(\Bar{\p} ^N) \rightharpoonup \xi \p^\frac{1}{p}$, so that by uniqueness of the limit, $\xi = \nabla L_\ell(\p) \p^{-\frac{1}{p}}$. Since we also have the convergence of the $L^q$ norm, we actually have strong convergence of the slope term.

We now prove the convergence of the discrete velocity field $\Bar{v}^N$. Just like in \cref{compactness_section}, by looking at the geodesic interpolation we can prove the following inequality :
\begin{align*}
    \int_{\Omega_T} \left|\Bar{v}_t ^N\right|^p \diff \Bar{\p}_t ^N \diff t \leq \int_{\Omega_T} \left|v_t \right|^p \diff \p \diff t,
\end{align*}
which implies that the sequence $\left(\Bar{v}^N (\Bar{\p}^N)^\frac{1}{p} \right)$ is bounded in $L^p _{t,x}$, so that up to extraction it weakly converges to some $\xi \in L^p _{t,x}$. Again like before we also have
\begin{align*}
    \int_{\Omega_T} \left|\Bar{v}_t ^N\Bar{\p}_t ^N\right| \diff x \diff t \leq T^\frac{1}{q}\left(\int_0 ^T \lVert \Bar{v}_t ^N \rVert ^p _{L^p (\Bar{\p}_t ^N)} \right)^\frac{1}{p} \leq C,
\end{align*}
so that the sequence $\left(\Bar{v}^N\Bar{\p}^N \right)$ weakly converges as a measure to some $E \in \mathcal{M}(\Omega)^d$. Using the convergence $\Bar{\p}^N \to \p$ and \cref{BenamouBrenier}, we deduce that there exists some $\hat{v}\in L^p(\p)$ such that $E=\hat{v}\p$ and 
\begin{align*}
    \int_{\Omega_T} \left|v\right|^p \diff \p \diff t \geq \liminf_{N \to \infty} \int_{\Omega_T} \left|\Bar{v}_t ^N \right|^p \diff \Bar{\p} \diff t \geq \int_{\Omega_T} \left|\hat{v}\right|^p \diff \p \diff t.
\end{align*}
Using arguments as in \cite{OTAM} section 9.3, we can find that $(\p,E)$ satisfy the continuity equation, so that since we chose $v$ to be the velocity field with minimum norm, we have $\hat{v}=v$ in $L^p(\p_t)$ for almost every $t$ and
\begin{align*}
    \int_{\Omega_T} \left|v\right|^p \diff \p \diff t =\int_{\Omega_T} \left|\hat{v}\right|^p \diff \p \diff t= \lim_{N \to \infty} \int_{\Omega_T} \left|\Bar{v}_t ^N\right|^p \diff \Bar{\p}_t ^N \diff t.
\end{align*}
We will now use the same convex trick to prove prove that $\xi =v \p^\frac{1}{p}$. We write $\omega(z,x)=\frac{|z|^p}{p} - \frac{|x|^p}{p} - x^{p-1} (z-x) \geq 0$, which enjoys the same properties as before, and $A_M= \left\{ |v|\leq M\right\}$, and we have
\begin{multline}
\label{trickforv}
    \int_{\Omega_T} \Bar{\p}_t ^N \frac{|\Bar{v}_t ^N|^p}{p} \diff x \diff t \geq \int_{A_M} \Bar{\p}_t ^N \frac{|v|^p}{p} \diff x \diff t + \int_{A_M} \Bar{\p}_t ^N |v|^{p-1} (\Bar{v}_t ^N - v) \diff x \diff t \\
    + \int_{A_M} \Bar{\p}_t ^N \omega (\Bar{v}_t ^N,v) \diff x \diff t.
\end{multline}
We know that $(\Bar{\p}_t ^N\Bar{v}_t ^N)$ is uniformly integrable, indeed, if $B \subset  \Omega\times [0,T]$ is measurable, then
\begin{align*}
    \int_B \Bar{\p}_t ^N \Bar{v}_t ^N \diff x \diff t  \leq \left(\int_B \Bar{\p}_t ^N |\Bar{v}_t ^N|^p \diff x \diff t \right)^\frac{1}{p} \left( \int_B \Bar{\p}_t ^N \diff x \diff t \right)^\frac{1}{q} \leq C \left( \int_B \Bar{\p}_t ^N \diff x \diff t\right)^\frac{1}{q}  , 
\end{align*}
so that the uniform integrability of $\Bar{\p}^N$ implies the unform integrability of $(\Bar{\p}^N\Bar{v}^N)$. Therefore, we conclude that $\Bar{\p}^N\Bar{v}^N \rightharpoonup \p v$ weakly in $L^1 _ {t,x}$ and
\begin{align*}
    \int_{A_M} \Bar{p}_t ^N\Bar{v}_t ^N|v|^{p-1}\diff x \diff t \to \int_{A_M} \p |v|^p \diff x \diff t
\end{align*}
Taking the limit in \cref{trickforv}, we therefore get
\begin{align*}
    \int_{A_M ^c} \p \frac{|v|^p}{p} \diff x \diff t \geq \limsup_{N \to \infty} \int_{A_M} \Bar{\p}_t ^N \omega(\Bar{v}_t ^N, v) \diff x \diff t.
\end{align*}
Again, the left hand side is decreasing in $M$ while the right hand side is increasing in $M$, so that we have
\begin{align*}
    0 = \limsup_{N \to \infty} \int_{\Omega_T} \Bar{\p}_t ^N\omega(\Bar{v}_t ^N,v) \diff x \diff t,
\end{align*}
and, up to extracting again, we find that $\Bar{v} \to v$ almost everywhere on on $\{\p > 0\}$. From this convergence and $\Bar{v} ^N(\Bar{\p}^N)^\frac{1}{p} \rightharpoonup \xi$, we deduce that $\xi = v \p^\frac{1}{p}$ on $\{\p >0 \}$.
Putting everything together,  we can write
\begin{align*}
    \int_{\Omega_T} |\xi|^p \diff x \diff t \leq \liminf_{N \to \infty} \int_{\Omega_T} |\Bar{v}_t ^N|^p \Bar{\p}_t ^N \diff x \diff x &=\int_{\Omega_T} |v|^p \p  \diff x \diff t \\
    &=\int_{\{\p >0\}} |v|^p \p  \diff x \diff t \\
    &=\int_{\{\p >0\}} |\xi|^p   \diff x \diff t \leq \int_{\Omega_T} |\xi|^p  \diff x \diff t,
\end{align*}
so that $\xi = v \p^\frac{1}{p}$ almost everywhere. In conclusion, we have proved that $\Bar{v}^N (\Bar{\p}^N)^\frac{1}{p} \to v \p^\frac{1}{p} $ in $L^p _{t,x}$, because we also have convergence of the norm. We can now finally pass to the limit in \cref{trucpasserlimite}, and we obtain 
\begin{align*}
     \mathscr{L}(\p_T)-\mathscr{L}(\p_0) \geq \int_{\Omega_T} \frac{\nabla L_\ell(\p)}{\p} \cdot  v \diff \p \diff t.
\end{align*}
Applying this result to the same curve while inverting the time variable, i.e.\ to the curve $\rho_t=\p_{T-t}$ (whose velocity vector is $-v_{T-t}$) we obtain 
\begin{align*}
     \mathscr{L}(\p_T)-\mathscr{L}(\p_0) = \int_{\Omega_T} \frac{\nabla L_\ell(\p)}{\p} v \diff \p \diff t.
\end{align*}
\end{proof}

In the previous proof we needed the following technical integration lemma.

\begin{lemma}
\label{sommes}
Let $f: [0,T] \to \R_+$ be a $L^1$ function. For all $\varepsilon >0$, there exists $0=t_0< t_1< t_2< \dots < t_{n+1}=T$  such that for all $i \in \left\{0, \dots ,n\right\}$, $t_{i+1} - t_i \leq \varepsilon$ and
\begin{align*}
    \left|\sum_{i=1} ^n f(t_i) (t_{i+1}- t_i) - \int_0 ^T f(x) \diff x \right| \leq \varepsilon.
\end{align*}
The $(t_i)_{i=1,\dots,n}$ can also be taken from a set $[0;T]\setminus A$ where $A$ has zero Lebesgue measure.
\end{lemma}

\begin{proof}
For the sake of clarity we will pick arbitrary $\varepsilon, w >0$ corresponding respectively to the maximum tolerated error in approximating the integral and the maximum mesh size, then one can set $w=\varepsilon$ at the end of the proof. 

By absolute continuity of the Lebesgue integral, there exists some $v > 0$ such that for all measurable subsets $E$ of $[0, T]$ with $\left|E\right| \leq v$ we have 
\begin{equation}
\label{eq1}
    \int_E f(x) \, \diff x\leq \frac{\varepsilon}{8}
\end{equation}
By the dominated convergence theorem, there exists some $M > 1$ such that 
\begin{align*}
    \left|\{x \in [0, T] ; f(x) > M\} \right| < \text{min}\left(\frac{v}{2}, \frac{w}{2}\right).
\end{align*}  Let $H=\left\{f(x) \leq M \right\}\setminus A$.
By the Lebesgue differentiation theorem, for a.e. $x \in H$, there exists $\delta(x)$ such that for all intervals $I$ containing $x$ of length less than $\delta(x)$, we have
\begin{equation}
    \label{eq2}
    \left|\frac{1}{|I|}\int_I f(y) \, dy  - f(x)\right| \leq \frac{\varepsilon}{8(T+1)}
\end{equation}
let $G$ be the (full measure) subset of $H$ such that the above holds.

For every $x \in G$, consider the closed intervals $[x, x + h]$, where $h$ ranges across all positive values such that $f(x + h) - f(x) \leq \frac{\varepsilon}{8}$, and $h \leq \text{min}(\frac{\varepsilon}{8M}, \frac{w}{2}, \delta(x))$. 

The above collection, ranging across $x$ forms a Vitali cover of the set $G$ (because otherwise the Lebesgue differentiation theorem would be contradicted). By Vitali's covering lemma, there exists a countable subcollection $I_i$ such that $G \setminus \cup_i I_i$ has measure $0$, and the $I_i$ are pairwise disjoint. By the Lebesgue dominated convergence theorem, there exists some $N > 0$ such that $\left|G \setminus \cup_{i = 1}^N I_i \right|\leq \text{min}(\frac{\varepsilon}{8M}, \frac{v}{2}, \frac{w}{2}).$ Now we write each $I_i$ as $(a_i, b_i)$, and relabel them so that they are in increasing order. We now take the partition given by the points $t_j$ to be $a_1, b_1, \dots, a_N, b_N$. We denote by $A_i, B_i$ the partition cells $(a_i, b_i)$ and $(b_i, a_{i+1})$, respectively, with the convention that $a_{N+1} =T$; and define $S(A_i)= f(a_i) (b_i - a_i)$ ; $S(B_i)=f(b_i)(a_{i+1} - b_i)$. Finally, we call $\mathcal R(f)$ the left Riemann sum of $f$ over this partition. We have:
\begin{align*}
    \left|\mathcal R(f) - \int_0^T f \, \diff x\right|  \leq \sum_{i = 1}^N \left|\int_{A_i} f \, \diff x - S(A_i)\right| + \sum_{j = 1}^{N} \left|\int_{B_j} f \, \diff x - S(B_j)\right| + \left|\int_{[0,a_1]} f \, \diff x \right|
\end{align*}
Using \cref{eq2}, we have:
\begin{equation*}
    \left|\int_{A_i} f \, \diff x - S(A_i) \right| \leq \left|A_i\right| \left(\frac{\varepsilon}{8(T+1)}\right)
\end{equation*}
Using \cref{eq1} and noticing that $\cup_{j = 1}^{N} B_j \subset [0, T] \setminus \cup_{i = 1}^N I_i$ whose measure is less than $v$ we have
\begin{align*}
    \sum_{j = 1}^{N} \left|\int_{B_j} f \, \diff x - S(B_j)\right| &\leq \left|\int_{\cup_{j = 1}^{N-1} B_j} f \, \diff x \right| + \sum_{j = 1}^{N-1} \left|S(B_j)\right| \\
    & \leq \frac{\varepsilon}{8} + \sum_{j = 1}^{N-1} \left|S(B_j)\right| \\
    & \leq \frac{\varepsilon}{8} +  \sum_{j = 1}^{N} \left|B_j\right| f(b_i) \\
    &\leq  \frac{\varepsilon}{8} + \sum_{j = 1}^{N} \left|B_j\right| (f(a_i) + \frac{\varepsilon}{8}) \\
    & \leq \frac{\varepsilon}{8} + (\sum_{j = 1}^{N} \left|B_j\right|)(M + \frac{\varepsilon}{8}) \\
    &= \frac{\varepsilon}{8} +  \frac{\varepsilon}{8M} (M + \frac{\varepsilon}{8}) \\
    &\leq \frac{\varepsilon}{8} + \frac{\varepsilon}{8} + \frac{\varepsilon^2}{16M} \\
    &= \frac{3\varepsilon}{8}
\end{align*}
Now for the last term, the domain of integration has measure less than $v$ by construction, and therefore
\begin{equation*}
    \left|\int_{[0, a_1]} f \, \diff x\right| \leq \frac{\varepsilon}{8}.
\end{equation*}
Thus putting all the above together, we have 
\begin{align*}
    \left|\mathcal R(f) - \int_0^T f \, \diff x\right| &\leq \frac{\varepsilon}{8(T+1)} \sum_{i = 1}^N \left|A_i\right| +  \frac{3\varepsilon}{8} + \frac{\varepsilon}{8} \\
    &\leq \frac{\varepsilon}{8(T+1)} T + \frac{3\varepsilon}{8} +  \frac{\varepsilon}{8} \\
    &\leq \varepsilon
\end{align*}
Finally, using that $\left| [0, T] \setminus \cup_{i = 1}^N I_i \right| \leq \left| H^c \right| + \left|G \setminus \cup_{i = 1}^N I_i \right| \leq w $ and the fact that the intervals are disjoint, we get that the mesh is of size at most $w$, and the first point of the partition is at most $w$ away from 0.
\end{proof}

{\bf Acknowledgments} The authors  were supported by the European Union via the ERC AdG 101054420 EYAWKAJKOS project.

\bibliographystyle{plain}
\bibliography{biblio}

\end{document}